\documentclass{amsart} 

\usepackage{amsthm, amssymb, amsmath, amscd}
\usepackage{eurosym}
\usepackage{amsfonts}
\usepackage{amsmath}
\usepackage{setspace}
\usepackage{bbm}
\usepackage{graphicx}

\usepackage[all,cmtip]{xy}

\usepackage{indentfirst}

\newtheorem{theorem}{Theorem}[section]

\newtheorem{cor}[theorem]{Corollary}

\newtheorem{lemma}[theorem]{Lemma}

\newtheorem{prop}[theorem]{Proposition}

\theoremstyle{definition}
\newtheorem{define}[theorem]{Definition}
\newtheorem{ex}[theorem]{Example}
\newtheorem{remark}[theorem]{Remark}

\newcommand{\ind}{\mathrm{ind}}

\newcommand{\field}[1]{\mathbb{#1}}

\newcommand{\zkz}{\field{Z}/k\field{Z}}
\newcommand{\C}{\field{C}}
\newcommand{\R}{\field{R}}
\newcommand{\Z}{\mathbbm{Z}}
\newcommand{\rz}{\field{R}/\field{Z}}

\begin{document}
\title[Realizing the analytic surgery group geometrically: Part I]{Realizing the analytic surgery group of Higson and Roe geometrically,\\ 
Part I: The geometric model}
\author{Robin J. Deeley, Magnus Goffeng}
\date{\today}
\begin{abstract}
We construct a geometric analog of the analytic surgery group of Higson and Roe for the assembly mapping for free actions of a group with values in a Banach algebra completion of the group algebra. We prove that the geometrically defined group, in analogy with the analytic surgery group, fits into a six term exact sequence with the assembly mapping and also discuss mappings with domain the geometric group. In particular, given two finite dimensional unitary representations of the same rank, we define a map in the spirit of $\eta$-type invariants from the geometric group (with respect to assembly for the full group $C^*$-algebra) to the real numbers. 
\end{abstract}

\maketitle

\section*{Introduction}
\noindent 
Higson and Roe begin their paper \cite{HReta} with the following statement:
 \vspace{0.20cm} \\
``Our aim is to place certain rigidity theorems for the relative eta-invariant into the context of Baum's geometric $K$-homology theory."
 \vspace{0.20cm}
\\
In the current paper, our goal is to take this idea to its natural conclusion; namely, we propose a definition of a geometric version of the analytic surgery group of Higson and Roe. Using this group, we obtain rigidity results for relative $\eta$-type invariants using only the framework of geometric cycles (rather than the combination of the geometric and analytic K-homology used in \cite{HReta}). In future work, we will link these purely geometrically defined invariants with classical and higher relative $\eta$-invariants.
\par
By a geometric analog, we mean a Baum-Douglas type construction, cf. \cite{BD,BD2}. The link between the analytic surgery group and classical surgery is the topic of the series of papers \cite{HRSur1, HRSur2, HRSur3}. The analytic surgery group is an invariant associated with a finitely generated discrete group $\Gamma$ and a choice $C^*_\epsilon(\Gamma)$ of a $C^*$-completion of the group ring $\C[\Gamma]$. In various geometric and analytic applications, the analytic surgery group can often be used as a domain for secondary invariants of the assembly mapping
$$\mu_\epsilon:K_*(B\Gamma)\to K_*(C^*_\epsilon(\Gamma))$$
as defined at the level of geometric cycles in (for example) \cite{Wal}. This assembly mapping coincides with the Baum-Connes assembly mapping whenever $\Gamma$ is torsion-free. In the setup of the Baum-Connes conjecture, the reduced completion $\epsilon={\bf red}$ and the full completion $\epsilon={\bf full}$ are the most well studied. Other completions have also been studied in \cite{baumguwill}. The veracity of the Baum-Connes conjecture for torsion-free discrete groups $\Gamma$ implies the vanishing of a large number of higher secondary invariants of free $\Gamma$-actions (see \cite{HReta, Kas, PS}); in most cases, the relevant completion is the full completion.

The motivation to carry out the construction of such a geometric analog comes from the philosophy behind Baum-Douglas models; by replacing analytic classes by fewer cycles consisting of geometric data, the cycles can be treated by geometric methods. The analytic surgery group of a discrete finitely generated group $\mathcal{S}^{an}_*(\Gamma)$ was in \cite{HR, HRSur3, HReta, Roe} constructed by means of coarse geometry as the $K$-theory of the $C^*$-algebra generated by $\Gamma$-invariant pseudolocal controlled operators. Following the philosophy of the Baum-Douglas model, which in \cite{BD} treats primary invariants, we treat a geometric model for the analytic surgery group in a way fitting with the secondary invariants coming from higher index theory. A virtue of the model is that it, without any major technical complications, lends itself to a generalization for assembly mappings with values in $K$-theory of Banach algebras allowing for more flexibility when proving vanishing results of secondary invariants.

This paper is the first in a series of papers. In forthcoming papers, we will apply the geometric model constructed in this paper to vanishing results for $\eta$-invariants, higher invariants, and relate our construction with Higson-Roe's analytic model. 
\\

To begin, we will give a short review the construction in \cite{Roe} (also see \cite[Chapter 12]{HR}, \cite{HRSur3} and \cite{HReta}).  Throughout, $\Gamma$ denotes a finitely generated group. Assume that $C^*_\epsilon(\Gamma)$ is some $C^*$-algebra completion of $\C[\Gamma]$. As in \cite{HRSur3}, $(\tilde{X},X,\gamma)$, will denote a $\Gamma$-presented space.  That is, $\tilde{X}$ is a proper geodesic metric space, $\gamma$ is a free and proper action of $\Gamma$ on $\tilde{X}$ and $X=\tilde{X}/\Gamma$.  We will assume that $X$ is compact; often we will require it to be a finite CW-complex.  In fact, the reader (if so inclined) could restrict to the case where $X$ is a manifold with fundamental group $\Gamma$ and $\tilde{X}$ is the universal cover of $X$.  Since the action is free, we have that
$K^{\Gamma}_*(\tilde{X}) \cong K_*(X)$. The assembly mapping that will be of interest to us takes the following form
$$\mu_{X,\epsilon} : K_*(X) \rightarrow K_*(pt;C^*_\epsilon(\Gamma)), \; (M,E,\varphi)\mapsto (M,E\otimes_\C \varphi^*\mathcal{L}_{C^*_\epsilon(\Gamma)})$$
where $\mathcal{L}_{C^*_\epsilon(\Gamma)}\times _\Gamma \tilde{X}$ denotes the Mishchenko bundle on $X$ associated with the completion $C^*_\epsilon(\Gamma)$ and the domain and range of this map are defined using Baum-Douglas type cycles (see Definition \ref{cycDef} below).  

The assembly map can also (see \cite[Section 12.5]{HR}) be constructed as follows.  Let $C^*_{\Gamma,\epsilon}(\tilde{X})$ denote the $C^*$-algebra generated by $\Gamma$-invariant locally compact controlled operators and $D^*_{\Gamma,\epsilon}(\tilde{X})$ denote the $C^*$-algebra generated by $\Gamma$-invariant pseudolocal controlled operators; both algebras are constructed on a Hilbert space which depends on the completion $\epsilon$.  Since the former is an ideal in the latter, we have the short exact sequence
$$0 \rightarrow C^*_{\Gamma,\epsilon}(X) \rightarrow D^*_{\Gamma,\epsilon}(X) \rightarrow D^*_{\Gamma,\epsilon}(X)/C^*_{\Gamma,\epsilon}(X) \rightarrow 0 $$
Applying the K-theory functor leads to the six-term exact sequence
\begin{center}
$\begin{CD}
K_0(D^*_{\Gamma,\epsilon}(X)/C^*_{\Gamma,\epsilon}(X)) @>>> K_1(C^*_{\Gamma,\epsilon}(X)) @>>>  K_1(D^*_{\Gamma,\epsilon}(X))\\
@AAA @. @VVV \\
 K_0(D^*_{\Gamma,\epsilon}(X)) @<<<  K_0(C^*_{\Gamma,\epsilon}(X)) @<<< K_1(D^*_{\Gamma,\epsilon}(X)/C^*_{\Gamma,\epsilon}(X))
\end{CD}$
\end{center}
Moreover, one can show (see \cite[Lemmas 12.5.3 and 12.5.4]{HR}) that 
\begin{eqnarray*}
K_p(C^*_{\Gamma,\epsilon}(X)) & \cong & K_p(C^*_\epsilon(\Gamma)) \\
K_{p+1}(D^*_{\Gamma,\epsilon}(X)/C^*_{\Gamma,\epsilon}(X)) & \cong & K_p(X)
\end{eqnarray*}
and that the boundary map in the six-term exact sequence above is given by the Baum-Connes assembly map.  Summarizing, we have the six-term exact sequence
\begin{center}
$\begin{CD}
K_0(X) @>\mu_{X,\epsilon}>> K_0(C^*_\epsilon(\Gamma)) @>>>  K_1(D^*_{\Gamma,\epsilon}(X))\\
@AAA @. @VVV \\
 K_0(D^*_{\Gamma,\epsilon}(X)) @<<<  K_1(C^*_\epsilon(\Gamma)) @<\mu_{X,\epsilon}<< K_1(X) 
\end{CD}$
\end{center}
As mentioned above, $K_p(X)$ and $K_p(C^*_\epsilon(\Gamma))\cong K_p(pt;C^*_\epsilon(\Gamma))$ admit realizations in terms of geometric cycles and the map $\mu_{X,\epsilon}$ can be defined at the level of cycles.  As such, one is led to consider the possiblity that $K_p(D^*_{\Gamma,\epsilon}(X))$ can be realized using a variant of the cycles considered by Baum and Douglas. 

The main goal of this paper is the construction of such a model. As in \cite{Dee1,DeeRZ}, the geometric model for the analytic surgery group is a relative cycle theory. As such, a cycle should consist of a cycle for $K_*^{geo}(X)$ that vanishes under the assembly mapping and a precise reason for its vanishing. For example, the class of a cycle $(M,E_\C,\varphi)$ in $K_*^{geo}(X)$ vanishes under assembly if there is a spin$^c$-manifold $W$ with boundary $M$, a locally trivial, finitely generated, projective $C^*_\epsilon(\Gamma)$-bundle $\mathcal{E}_{C^*_\epsilon(\Gamma)}\to W$ and an isomorphism 
\begin{equation}
\label{alphaececstar}
\alpha:E_\C\otimes \varphi^*\mathcal{L}_{C^*_\epsilon(\Gamma)}\xrightarrow{\sim} \mathcal{E}|_M
\end{equation} 
of $C^*_\epsilon(\Gamma)$-bundles on $M$. The data $(W,(\mathcal{E}_{C^*_\epsilon(\Gamma)},E_\C,\alpha),\varphi)$ is the prototypical example of a cycle in the geometric model for the analytic surgery group. 

Our construction is at present only an analogy, since we do not construct an explicit isomorphism between our geometric group and $K_*(D^*_{\Gamma,\epsilon}(X))$. However, the geometric group fits into the required exact sequence (see Theorem \ref{sixTerExtSeq}). We do have an explicit map which we believe is an isomorphism, the proof of it being well defined uses the results of \cite{PSrhoIndSign} heavily. We will consider this question in a forthcoming paper. Despite this fact, results concerning the vanishing/homotopy invariance of $\eta$-type invariants can be obtained purely from the geometric group.
\par

\subsection*{The content of the paper}  
We start the paper with some preliminaries in Section \ref{sectionpreliminaries}. This includes recalling geometric models for $K$-homology of topological spaces with coefficients in a Banach algebra, a subject to which we found no detailed reference; the proofs of the relevant results are included in Appendix \ref{banachktheory}. We also define the geometric assembly mapping for free actions of a discrete group $\Gamma$ with values in the $K$-theory of a Banach algebra completion $\mathcal{A}(\Gamma)$ of $\C[\Gamma]$. When $\Gamma$ is torsion-free and the completion $\mathcal{A}(\Gamma)$ is unconditional, this assembly mapping coincides with the assembly mapping for proper actions constructed in Lafforgue's $KK^{ban}$-theory assuming that an embedding of $\mathcal{A}(\Gamma)$ into a $C^*$-algebra completion of $\C[\Gamma]$ induces an isomorphism on $K$-theory. Furthermore, in Subsection \ref{subsectionrelativektheory} and \ref{subsectionmodelforrelativektheory}, we construct a relative $K$-theory group for the assembly mapping on $K$-theory. This is motivated by the need for a $K$-theory group to describe the vector bundle data in the geometric cycles (e.g., the data $(\mathcal{E}_{C^*_\epsilon(\Gamma)},E_\C,\alpha)$ as above in \eqref{alphaececstar}).

In Section \ref{sectiongeometriccycles}, we define the geometric cycles. These cycles are defined in a general setup depending on the following data: a locally compact Hausdorff space $X$, a unital Banach algebra $B$ and a locally trivial bundle $\mathcal{L}\to X$ of finitely generated projective $B$-modules. A cycle (with respect to this input) is $(W,\xi,f)$ where $W$ is a spin$^c$-manifold with boundary, $f:\partial W\to X$ and $\xi$ is a relative $K$-theory class for the bundle $\mathcal{L}$ analogously to \eqref{alphaececstar}. These cycles modulo the standard type of relation form a $\Z/2$-graded abelian group $\mathcal{S}_*^{geo}(X,\mathcal{L})$. The main result of this section, contained in Theorem \ref{sixTerExtSeq}, states that  $\mathcal{S}_*^{geo}(X,\mathcal{L})$ fits into a six term exact sequence with the geometric assembly mapping
$$\begin{CD}
K_0^{geo}(X) @>\mu_\mathcal{L}>> K_0^{geo}(pt;B) @>r>>  \mathcal{S}^{geo}_0(X;\mathcal{L})\\
@AA\delta A @. @VV\delta V \\
 \mathcal{S}^{geo}_1(X;\mathcal{L}) @<r<<  K_1^{geo}(pt;B) @<\mu_\mathcal{L}<< K_1^{geo}(X).
\end{CD}$$
If $\mathcal{A}(\Gamma)$ is a Banach algebra completion of $\C[\Gamma]$ and $\mathcal{L}_{\mathcal{A}(\Gamma)}$ is the Mishchenko bundle defined from $\mathcal{A}(\Gamma)$; we use the notation $\mathcal{S}^{geo}_*(\Gamma,\mathcal{A}):=\mathcal{S}^{geo}_*(B\Gamma,\mathcal{L}_{\mathcal{A}(\Gamma)})$. A direct advantage is that in the framework of geometric cycles, it is straight forward to construct an exterior product
\[K_*^{geo}(B\Gamma_1)\times \mathcal{S}^{geo}_*(\Gamma_2,\mathcal{A}'')\to \mathcal{S}^{geo}_*(\Gamma_1\times \Gamma_2,\mathcal{A}),\]
whenever $\Gamma_1$ and $\Gamma_2$ are two discrete groups and $\mathcal{A}''$ and $\mathcal{A}$ are suitable Banach algebra completions of $\Gamma_2$ respectively $\Gamma_1\times \Gamma_2$. See more in Proposition \ref{productprop}. This exterior product was constructed in the analytic framework in \cite{Seg}.

The cycles in the geometric model for analytic surgery are of a very special form allowing for the construction of various geometric invariants. We consider the construction of such invariants in Section  \ref{subsectionhigherrho}. The general set up is somewhat involved, but a particular case is the following one. Here we work with the full group $C^*$-algebra and suppose that we are given two finite dimensional unitary representations of the same rank. These two representations give a cocycle in $K^1(B\Gamma;\rz)$ and from each a map from $K_*(C^*_{{\bf full}}(\Gamma))$ to $K_*(\field{C})$.
In section \ref{subsectionhigherrho}, we construct a mapping $\mathcal{S}^{geo}_*(\Gamma,C^*_{{\bf full}})\to \mathbb R$ that fits into a commuting diagram with exact rows:
\begin{center}
\scriptsize
$\begin{CD}
 K_0^{geo}(B\Gamma) @>\mu_{{\bf full}}>> K_0^{geo}(pt;C^*_{{\bf full}}(\Gamma))  @>>> \mathcal{S}^{geo}_0(\Gamma;C^*_{{\bf full}}) @>>> K_1^{geo}(B\Gamma) @>\mu_{{\bf full}}>> K_1^{geo}(pt;C^*_{{\bf full}}(\Gamma))   \\
 @VVV  @VVV @VVV @VVV @VVV  \\
  0 @>>> \field{Z} @>>> \mathbb R @>>> \rz @>>> 0
\end{CD}$
\normalsize
\end{center} 
Here, the mapping on $K_*^{geo}(B\Gamma)$ is defined by pairing with the cocycle in $K^1(B\Gamma;\rz)$ associated with the two representations and the map on $ K_*(C^*_{{\bf full}}(\Gamma))$ is given by the difference of the maps induced from the representations. Again, the reader should note the analogy with work of Higson and Roe in \cite{HReta}. However, in Section \ref{subsectionhigherrho}, we construct a higher version of these mappings which relates the six-term exact sequence assoicated with the assembly map with the Bockstein sequence in $K$-theory associated to the short exact sequence of groups $ 0 \rightarrow \field{Z}\rightarrow \mathbb R \rightarrow \rz \rightarrow 0$.

\section{Preliminaries}
\label{sectionpreliminaries}

\subsection{Baum-Connes mappings}
\label{section2.a}

A central object of study in this paper is the $K$-theory of Banach algebras associated with discrete finitely generated groups $\Gamma$. As we will need to work with fairly general Banach algebras, we will make a notational distinction between general Banach algebras and Banach algebra closures of $\C[\Gamma]$, $B$ for the former and $\mathcal{A}(\Gamma)$ for the latter. For the convenience of the reader we recall the notion of a geometric cycle.

\begin{define}[cf. \cite{BD}]
\label{geometriccyclesbacoeff}
Assume that $X$ is a locally compact Hausdorff space and $B$ is a unital Banach algebra. A geometric cycle with coefficients in $B$ is a triple $(M,\mathcal{E}_B,\varphi)$ where 
\begin{enumerate}
\item $M$ is a compact smooth spin$^c$-manifold.
\item $\mathcal{E}_B\to M$ is a smooth locally trivial bundle of finitely generated projective $B$-bundle.
\item $\varphi:M\to X$ is a continuous mapping.
\end{enumerate}
The set of isomorphism classes of geometric cycles with coefficients in $B$ is $\Z/2\Z$-graded by the parity of the dimension of $M$ and is equipped with an equivalence relation generated by disjoint union/direct sum, bordism and vector bundle modification. The set of equivalence classes forms a $\Z/2\Z$-graded abelian group denoted by $K_*^{geo}(X;B)$. The group $K_*^{geo}(X;B)$ depends covariantly on $X$ and covariantly on $B$ under homomorphisms of Banach algebras. If $B=\C$, we simply write $K_*^{geo}(X):=K_*^{geo}(X;\C)$. If $X$ is non-compact we note that 
$$K_*^{geo}(X;B)=\varinjlim_{X'\subseteq X\,\mbox{compact}} K_*^{geo}(X';B).$$
\end{define}

The smoothness assumption on the bundle $\mathcal{E}_B\to M$ can be lifted after using a similar argument as in the proof of Theorem $3.14$ of \cite{Sch}.

\begin{remark}
\label{anassem}
The definition above of $K_*^{geo}(X)$ coincides with the usual definition. We prove in the appendix that $K_*^{geo}(-;B)$ forms a generalized homology theory. It is unclear to which extent there exists a natural mapping 
$$K_*^{geo}(X;B)\to KK^{ban}(C_0(X),B),$$
defined analytically. Here $KK^{ban}$ denotes Lafforgue's $KK$-theory for Banach algebras. We provide an abstract condition on $B$ in Appendix \ref{banachktheory} for this mapping to exist and produce an isomorphism for all finite $CW$-complexes $X$. 

In general, there exists a topological index mapping defining an isomorphism $\ind_B^t:K_*^{geo}(pt;B)\xrightarrow{\sim} K_*(B)$ as follows. After replacing $B$ with $C(S^1,B)$ and using Bott periodicity, it suffices to consider $*=0$. If $(M,\mathcal{E}_B)$ is a cycle for $K_0^{geo}(pt;B)$, we choose a spin$^c$-embedding $i:M\hookrightarrow \R^{2N}$, for some large enough $N$; the existence of such an embedding is guaranteed by Lemma $2.1$ of \cite{BOOSW}. We let $\beta_N:K_0(C_0(\R^{2N};B))\to K_0(B)$ denote the Bott periodicity mapping. The mapping $(M,\mathcal{E}_B)\mapsto \beta_N i_![\mathcal{E}_B]$ produces the sought after isomorphism (for more details see Lemma \ref{analyticassemblyc}). 
\end{remark}

\begin{define}[Assembly along $B$-bundles]
Assume that $X$ is a locally compact Hausdorff space, $B$ is a Banach algebra and $\mathcal{L}\to X$ is a locally trivial bundle of projective finitely generated $B$-bundles. We define $\mu_\mathcal{L}:K_*^{geo}(X)\to K_*^{geo}(pt;B)$ by 
$$\mu_{\mathcal{L}}:K_*^{geo}(X)\to K_*^{geo}(pt;B), \;(M,E_\C,\varphi)\mapsto (M,E_\C\otimes \varphi^*\mathcal{L}).$$
\end{define}

The example of assembly which is absolutely central to this paper comes from discrete groups. Associated with a discrete finitely generated group $\Gamma$ there is a classifying space for free actions $B\Gamma$, which can be chosen to be a $CW$-complex. The universal property of $B\Gamma$ states that there is a one-to-one correspondence between free $\Gamma$-actions on a topological space $\tilde{X}$ such that $\tilde{X}/\Gamma=X$ and homotopy classes of mappings $u:X\to B\Gamma$. Whenever $\mathcal{A}(\Gamma)$ is a Banach algebra closure of $\C[\Gamma]$ such that the $\Gamma$-action on $\C[\Gamma]$ extends to a continuous operation on $\mathcal{A}(\Gamma)$ we can construct a Mishchenko bundle 
\[\mathcal{L}_{\mathcal{A}(\Gamma)}:=E\Gamma\times _\Gamma \mathcal{A}(\Gamma)\to B\Gamma.\]
If $u:X\to B\Gamma$ is continuous and $X$ is a smooth manifold, $u^*\mathcal{L}_{\mathcal{A}(\Gamma)}$ is a flat bundle of $\mathcal{A}(\Gamma)$-modules. To shorten notation, we define $\mu_{\mathcal{A}(\Gamma)}:=\mu_{\mathcal{L}_{\mathcal{A}(\Gamma)}}$.

\begin{remark}
In the case when $\Gamma$ is torsion-free, $\mu_{C^*_{\bf red}(\Gamma)}:K_*^{geo}(B\Gamma)\to K_*(C^*_{\bf red}(\Gamma))$ coincides with the Baum-Connes assembly mapping, see \cite[Corollary $2.16$]{valettepaper} or the more detailed proof in \cite{land}. Such a construction can be carried out by means of a descent construction for $X\subseteq B\Gamma$ compact, corresponding to a $\Gamma$-compact $\tilde{X}\subseteq E\Gamma$;
\begin{align*}
K_*(X)\cong& K_*^\Gamma(\tilde{X})=KK^\Gamma(C_0(\tilde{X}),\C)\\
\to &KK(C_0(\tilde{X})\rtimes_{\bf red} \Gamma,C^*_{\bf red}(\Gamma)) \to KK(\C,C^*_{\bf red}(\Gamma))=K_*(C^*_{\bf red}(\Gamma)).
\end{align*}
The isomorphism is the natural one arising from that the $\Gamma$-action on $\tilde{X}$ is free, the second mapping is the descent mapping and the third mapping is the Kasparov product with a distinguished element $[\theta_X]\in K_0(C_0(\tilde{X})\rtimes _{\bf red}\Gamma)$, see more in for instance \cite{valettebook,valettepaper}. Such a construction can in fact be carried out for any $C^*$-algebra completion $C^*_\epsilon(\Gamma)$ of $\C[\Gamma]$ and coincides with $\mu_{C^*_\epsilon(\Gamma)}$. A similar descent construction can be carried out using Lafforgue's $KK^{ban}$ for any Banach algebra $\mathcal{A}(\Gamma)$ which is a completion of $\C[\Gamma]$ in an unconditional norm, see more in \cite{laffinventiones,valettebook}. This produces an assembly mapping
$$\mu_{\mathcal{A}(\Gamma)}^{ban}:K_*(B\Gamma)\to K_*(\mathcal{A}(\Gamma)).$$
It is to the authors' knowledge not known whether $\mu_{\mathcal{A}(\Gamma)}=\mu_{\mathcal{A}(\Gamma)}^{ban}$ in general (see \cite[Introduction]{Dal}). 
\end{remark}

\begin{prop}
\label{mumuban}
Assume that $\Gamma$ is a torsion-free group, $\mathcal{A}(\Gamma)$ and $C^*_\epsilon(\Gamma)$ is an unconditional Banach algebra completion respectively a $C^*$-algebra completion of $\C[\Gamma]$ and that there is a continuous extension $j_{\mathcal{A},\epsilon}:\mathcal{A}(\Gamma)\to C^*_\epsilon(\Gamma)$ of the identity mapping on $\C[\Gamma]$ inducing an isomorphism $K_*(\mathcal{A}(\Gamma))\cong K_*(C^*_\epsilon(\Gamma))$, then $\mu_{\mathcal{A}(\Gamma)}=\mu_{\mathcal{A}(\Gamma)}^{ban}$.
\end{prop}

\begin{proof}
Since $j_{\mathcal{A},\epsilon}$ extends the identity mapping on $\C[\Gamma]$, $(j_{\mathcal{A},\epsilon})_*\mathcal{L}_\mathcal{A}=\mathcal{L}_\epsilon$. Thus, the Proposition follows from functoriality as
$$K_*(j_{\mathcal{A},\epsilon})\circ \mu_{\mathcal{A}(\Gamma)}=\mu_{C^*_\epsilon(\Gamma)}= K_*(j_{\mathcal{A},\epsilon})\circ \mu_{\mathcal{A}(\Gamma)}^{ban}$$
\end{proof}

Let us end this subsection with some further remarks about the assembly mapping for torsion-free groups. The Baum-Connes conjecture states that $\mu_{C^*_{\bf red}(\Gamma)}$ is an isomorphism, this is a well studied problem to which there to date are no known counterexamples. Motivated by this we say that the Banach algebra completion $\mathcal{A}(\Gamma)$ has the Baum-Connes property if $\mu_{\mathcal{A}(\Gamma)}$ is an isomorphism. If $C^*_\epsilon(\Gamma)$ is a $C^*$-completion of $\C[\Gamma]$ we say that $\Gamma$ has the $\epsilon$-Baum-Connes property if $\mu_{C^*_\epsilon(\Gamma)}$ is an isomorphism. 

The reader should note that the reason for us to consider rather general Banach algebra completions of $\C[\Gamma]$ is the lack of functoriality of the reduced $C^*$-algebra completion $C^*_{\bf red}(\Gamma)$ that appears in the Baum-Connes conjecture. If $\Gamma$ is amenable, or more generally $K$-amenable, the surjection $C^*_{\bf full}(\Gamma)\to C^*_{\bf red}(\Gamma)$, from the universal $C^*$-algebra completion $C^*_{\bf full}(\Gamma)$, induces an isomorphism $K_*(C^*_{\bf full}(\Gamma))\cong K_*(C^*_{\bf red}(\Gamma))$. Since $C^*_{\bf full}(\Gamma)$ is functorial, it renders the usage of Banach algebras for $K$-amenable groups unnecessary. In general, a natural candidate for a Banach algebra is $\ell^1(\Gamma)$ which has good functoriality properties. The Bost conjecture is the natural analog of the Baum-Connes conjecture in the case of $\ell^1(\Gamma)$; that is, it conjectures that the $\ell^1(\Gamma)$-assembly map is an isomorphism. There are no known counterexamples to the Bost conjecture to date. Note that if $\Gamma$ satisfies the Bost conjecture and has the $\epsilon$-Baum-Connes property for some $\epsilon$ then $K_*(j_{\ell^1,\epsilon})$ is an isomorphism and $\mu_{\ell^1}=\mu_{\ell^1}^{ban}$ by Proposition \ref{mumuban}.

As remarked above, the unconditional completions of $\C[\Gamma]$ (see more in \cite{valettebook}) plays a special role as the assembly mapping can be defined on the level of $KK^{ban}$ also for coefficients. The simplest example of an unconditional Banach algebra completion of $\C[\Gamma]$ is $\ell^1(\Gamma)$. Other examples include the Sobolev space $H^s_L(\Gamma)$, for $s$ large enough, with respect to a length function $L$ on a discrete group $\Gamma$ with property $(RD)$.

\subsection{Relative K-theory considerations}
\label{subsectionrelativektheory}

The setup we will use for relative $K$-theory groups is a compact Hausdorff space $X$, a unital Banach algebra $B$ and a $B$-bundle $\mathcal{L}\to X$, i.e., a locally trivial bundle of finitely generated projective right $B$-modules. The motivation for considering Banach algebras comes from Subsection \ref{section2.a} as the example of importance later in this paper is $X= B\Gamma$ and $\mathcal{L}:=E\Gamma\times_\Gamma B$ where $B$ is $C^*_{\bf full}(\Gamma)$, $C^*_{\bf red}(\Gamma)$, $\ell^1(\Gamma)$ or more generally any Banach algebra completion of $\C[\Gamma]$. The construction here is based on results in \cite{Kar} (also see \cite{DeeRZ}). For a topological space $Y$ and a Banach algebra $D$ we let $K^*(Y;D)$ denote the $K$-theory of the Banach algebra $C_0(Y,D)$.

Let $Z$ be a compact Hausdorff space and $Q \subseteq Z$ a closed subspace. Assume further that there is a fixed continuous map $g:Q\rightarrow X$.  We will consider two natural KK-theory elements:  The first is given by the natural $*$-homomorphism, $C(Z) \rightarrow C(Q)$; the second is given by the map $\mu_{g,\mathcal{L}}:K^0(Q) \rightarrow K^0(Q;B)$ given by $F \mapsto F\otimes g^*(\mathcal{L})$.  From these elements, we obtain a map on K-theory given by 
\begin{align*}
\sigma: K^0(Z;B)\oplus K^0(Q) &\rightarrow K^0(Q;B),\\
(\mathcal{E}_B,F_\field{C})&\mapsto [\mathcal{E}_B|_Q]-[F_\field{C}\otimes g^*\mathcal{L}].
\end{align*}

\begin{remark}
At this point we note that if $A$ and $B$ are $C^*$-algebras an analogous construction produces a mapping $K^0(Z;B)\oplus K^0(Q;A) \rightarrow K^0(Q;B)$ if $\mathcal{L}$ is an $A-B$-Hilbert bimodule bundle over $X$. Our geometric model as well as the geometric surgery sequence exists also in this setting. 
\end{remark}

\begin{define}
Assume that $B$ is a $C^*$-algebra. Let $\tilde{\sigma}$ be a $*$-homomorphism in the Cuntz picture of KK-theory whose KK-theory class gives the map $\sigma$.  Denote $K^*(C_{\tilde{\sigma}})$ by $K^*(Z,Q;\mu_\mathcal{L})$ where $C_{\tilde{\sigma}}$ is the mapping cone of $\tilde{\sigma}$.
\end{define}

By construction, we have the following exact sequence:
\begin{equation}
\label{KthExtSeq}
\begin{CD}
K^0(Z,Q;\mu_\mathcal{L}) @>>> K^0(Z;B)\oplus K^0(Q) @>\sigma>>  K^0(Q;B) \\
@AAA @. @VVV \\
 K^1(Q;B) @<\sigma<<  K^1(Q) \oplus K^1(Z;B) @<<< K^1(Z,Q;\mu_\mathcal{L}) 
\end{CD}
\end{equation}
Given a class $[\xi] \in K^0(Z,Q;\mu_\mathcal{L})$, let $[\xi_B]$ (respectively $[\xi_{\field{C}}]$) denote its image under the map $K^0(Z,Q;\mu_\mathcal{L}) \rightarrow K^0(Z;B))$ (respectively $K^0(Z,Q;\mu_\mathcal{L}) \rightarrow K^0(Q)$). It is clear from this exact sequence that $K^*(Z,Q;\mu_\mathcal{L})$, up to isomorphism, does not depend on the choice of $KK$-element $\tilde{\sigma}$.

\subsection{A model for $K^*(Z,Q;\mu_\mathcal{L})$}
\label{subsectionmodelforrelativektheory}

We will now turn to an explicit model for $K^0(Z,Q;\mu_\mathcal{L})$ for a general Banach algebra $B$. The analogous model for $K^1(Z,Q;\mu_\mathcal{L})$ can be found by, for instance, using a suspension argument.

\begin{define}
\label{relativektheorycycle}
A $\mu_\mathcal{L}$-relative $K$-theory cocycle over $(Z,Q,g)$ is a quintuple $\xi=(\mathcal{E}_B,\mathcal{E}'_B,E_\C,E'_\C,\alpha)$ where
\begin{enumerate}
\item $\mathcal{E}_B$ and $\mathcal{E}_B'$ are $B$-bundles on $Z$.
\item $E_\C$ and $E'_\C$ are vector bundles on $Q$.
\item $\alpha$ is an isomorphism of $B$-bundles 
$$\alpha:\mathcal{E}_B|_Q\oplus (E_\C' \otimes_\C g^*\mathcal{L})\xrightarrow{\sim} \mathcal{E}_B'|_Q  \oplus(E_\field{C}\otimes_\field{C} g^*\mathcal{L})$$
\end{enumerate}
An isomorphism of two $\mu_\mathcal{L}$-relative $K$-theory cocycles $\xi_1=(\mathcal{E}_B,\mathcal{E}'_B,E_\C,E'_\C,\alpha)$ and $\xi_2=(\mathcal{F}_B,\mathcal{F}'_B,F_\C,F'_\C,\beta)$ is a quadruple $\Phi=(\Phi_B,\Phi_B',\Phi_\C,\Phi_\C')$ of isomorphisms over $Z$ respectively $Q$
$$\Phi_B:\mathcal{E}_B\to \mathcal{F}_B,\;\Phi_B':\mathcal{E}_B'\to \mathcal{F}_B',\;\Phi_\C:E_\field{C}\to F_\field{C}\mbox{   and    } \Phi_\C':E_\field{C}'\to F_\field{C}'$$ 
such that the following diagram commutes:
\begin{center}
$\begin{CD}
\mathcal{E}_B|_Q\oplus (E_\C' \otimes_\C g^*\mathcal{L}) @>\alpha>>  \mathcal{E}_B'|_Q\oplus (E_\field{C}\otimes_\field{C} g^*\mathcal{L})   \\
 @V\Phi_B|_Q\oplus (\Phi_\C'\otimes \mathrm{id}_{g^*\mathcal{L}})  VV  @VV\Phi_B' \oplus (\Phi_\C\otimes \mathrm{id}_{g^*\mathcal{L}})  V  \\
 \mathcal{F}_B|_Q\oplus (F_\C' \otimes_\C g^*\mathcal{L}) @>\beta>>\mathcal{F}_B'|_Q\oplus (F_\field{C}\otimes_\field{C} g^*\mathcal{L})
\end{CD}$
\end{center}
If $Z$ is a manifold, and $Q$ a sub manifold, we say that $\xi$ is smooth if all the data it contains is smooth. A $\mu_\mathcal{L}$-relative $K$-theory cocycle $\xi_0=(\mathcal{E}_B,0,E_\C,0,\alpha)$ is called an \emph{easy} cocycle. 
We define the opposite cocycle of $\xi=(\mathcal{E}_B,\mathcal{E}'_B,E_\C,E'_\C,\alpha)$ as
$$\xi^{op}:=(\mathcal{E}_B',\mathcal{E}_B,E'_\C,E_\C,\alpha^{-1}).$$
\end{define}

\begin{remark}
The complicated condition $(3)$ in Definition \ref{relativektheorycycle} is due to the fact that $g^*\mathcal{L}$ need not extend to all of $Z$. It is needed to obtain all cocycles in relative $K$-theory groups. Condition $(3)$ is, after possibly stabilizing the bundles in their respective fashion, equivalent to the following relation in $K^0(Q;B)$
$$[\mathcal{E}_B|_Q]-[\mathcal{E}'_B|_Q]=[E_\field{C}\otimes_\field{C}g^*\mathcal{L}]-[E_\C'\otimes_\field{C}g^*\mathcal{L}].$$ 
It suffices to consider easy cocycles if $\mathcal{L}$ is trivial (for example, see the geometric models considered in \cite{DeeRZ}).
\end{remark}

The set of isomorphism classes of $\mu_\mathcal{L}$-relative $K$-theory cocycles over $(Z,Q,g)$ can be equipped with a semigroup operation:
\begin{align*}
(\mathcal{E}_B,\mathcal{E}'_B,E_\C,E'_\C,\alpha)+&(\mathcal{F}_B,\mathcal{F}'_B,F_\C,F'_\C,\beta):=\\
&(\mathcal{E}_B\oplus\mathcal{F}_B,\mathcal{E}'_B\oplus \mathcal{F}'_B,E_\C\oplus F_\C,E'_\C\oplus F_\C',\alpha\oplus \beta).
\end{align*}
It is clear that the operation induces an abelian semigroup structure on the set of isomorphism classes of $\mu_\mathcal{L}$-relative $K$-theory cocycles. 

\begin{define}[Degenerate cocycles]
\label{defofdeg}
We say that a $\mu_\mathcal{L}$-relative $K$-theory cocycle $\xi$ over $(Z,Q,g)$ is degenerate if it is a direct sum of $\mu_\mathcal{L}$-relative $K$-theory cocycles isomorphic to one of the following two forms
\begin{enumerate}
\item $(\mathcal{E}_B,\mathcal{E}'_B,E_\C,E_\C', \alpha)$ where 
$$\alpha=\alpha_0\oplus \alpha_1:\mathcal{E}_B|_Q\oplus (E_\C' \otimes_\C g^*\mathcal{L})\xrightarrow{\sim} \mathcal{E}_B'|_Q\oplus(E_\field{C}\otimes_\field{C} g^*\mathcal{L}) $$ 
and $\alpha_0$ extends to an isomorphism $\mathcal{E}_B\xrightarrow{\sim} \mathcal{E}_B'$ on $Z$ and $\alpha_1=\alpha_1'\otimes \mathrm{id}_{g^*\mathcal{L}}$ for an isomorphism $\alpha_1':E_\C'\to E_\C$.
\item $\xi_0\oplus \xi^{op}_0$ for a $\mu_\mathcal{L}$-relative $K$-theory cocycle $\xi_0$.
\end{enumerate}
\end{define}

It holds by construction that the set of degenerate cocycles is closed under direct sum.

\begin{prop}
\label{degeneratestabilization}
If $\xi$ is a $\mu_\mathcal{L}$-relative $K$-theory cocycle, there is a degenerate cocycle $\xi'$ such that $\xi\oplus \xi'$�is isomorphic to a cocycle of the form $(\mathcal{E}_B,Z\times B^n,E_\C,Q\times \C^k, \alpha)$ for some $n,k$.
\end{prop}

\begin{proof}
Follows from the fact that any $B$-bundle and vector bundle admits a complementary bundle.
\end{proof}

\begin{remark}
\label{complicatedremark}
There is an equivalent approach to the complicated condition $(3)$ of Definition \ref{relativektheorycycle}. We assume that $X$ is compact allowing us to choose a complementary bundle $\mathcal{L}_\perp\to X$ to $\mathcal{L}$ in a trivial bundle $X\times B^N$ and fix an isomorphism $\gamma:\mathcal{L}\oplus \mathcal{L}_\perp\xrightarrow{\sim} X\times B^N$. Using Proposition \ref{degeneratestabilization}, the $\mu_\mathcal{L}$-relative $K$-theory cocycles stand in one-to-one correspondence, up to stabilization by degenerate cocycles, with triples $(\mathcal{E}_B,E_\field{C},\tilde{\alpha})$ satisfying the conditions of Definition \ref{relativektheorycycle} but with the complicated condition $(3)$ replaced by that 
\begin{enumerate}
\item[(3')] $\tilde{\alpha}$ is an isomorphism of $B$-bundles 
$$\tilde{\alpha}:\mathcal{E}_B|_Q\oplus Q\times B^m \xrightarrow{\sim} (E_\field{C}\otimes_\field{C} g^*\mathcal{L})\oplus g^*\mathcal{L}^k_\perp \oplus Q\times B^n$$ 
for some $m,n,k\in \mathbbm{N}$.
\end{enumerate}
\end{remark}

For two isomorphic $B$-bundles $\mathcal{F},\mathcal{F}'\to Q$, we let $\mathrm{Iso}_B(\mathcal{F},\mathcal{F}')\to Q$ the locally trivial bundle of isomorphisms.

\begin{define}
We say that two $\mu_\mathcal{L}$-relative $K$-theory cocycles $\xi=(\mathcal{E}_B,\mathcal{E}'_B,E_\C,E_\C', \alpha)$ and $\xi'=(\mathcal{E}_B,\mathcal{E}'_B,E_\C,E_\C', \alpha')$ are homotopic if there is an 
$$\tilde{\alpha}\in C(Q\times [0,1], \pi^*\mathrm{Iso}_B(\mathcal{E}_B|_Q\oplus E_\C'\otimes g^*\mathcal{L},\mathcal{E}_B'|_Q\oplus E_\C\otimes g^*\mathcal{L})),$$ 
where $\pi:Q\times [0,1]\to Q$ denotes the projection, such that $\tilde{\alpha}|_{Q\times \{0\}}=\alpha$ and $\tilde{\alpha}|_{Q\times \{1\}}=\alpha'$. 
\end{define}

There is an obvious extension of the notion of homotopy to isomorphism classes of two $\mu_\mathcal{L}$-relative $K$-theory cocycles. The following characterization of homotopy of $\mu_\mathcal{L}$-relative $K$-theory cocycles follows from standard approximation arguments.

\begin{prop}
\label{homotopyprop}
Two isomorphism classes of $\mu_\mathcal{L}$-relative $K$-theory cocycles $\xi$ and $\xi'$ are homotopic if and only if there exists a $\mu_\mathcal{L}$-relative $K$-theory cocycle $\tilde{\xi}$ for $(Z\times [0,1],Q\times [0,1], g\circ \pi)$ such that 
$$\tilde{\xi}|_{(Z\times \{0\},Q\times \{0\}, g)}\cong \xi\quad\mbox{and}\quad \tilde{\xi}|_{(Z\times \{1\},Q\times \{1\}, g)}\cong \xi'.$$
If $Z$ is a manifold, and $Q$ a sub manifold, and $\xi$ and $\xi'$ are smooth we can in this case take $\tilde{\xi}$ to be smooth.
\end{prop}

\begin{define}
Let $Gr(Z,Q;\mu_\mathcal{L})$ denote the Grothendieck group of the semigroup of isomorphism classes of $\mu_\mathcal{L}$-relative $K$-theory cocycles over $(Z,Q,g)$. We define $Gr_0(Z,Q;\mu_\mathcal{L})\subseteq Gr(Z,Q;\mu_\mathcal{L})$ as the subgroup generated by the set of isomorphism classes of degenerate $\mu_\mathcal{L}$-relative $K$-theory cocycles. For a cocycle $\xi$ we will denote its class in $Gr(Z,Q;\mu_\mathcal{L})$ by $[\xi]_0$. We define $K^0(Z,Q;\mu_\mathcal{L})$ as the set of equivalence classes of elements from $Gr(Z,Q;\mu_\mathcal{L})$ under the relation generated by homotopy equivalence and degenerate equivalence;
$$\xi\sim_{deg} \xi'\Leftrightarrow\exists \xi_0\in Gr_0(Z,Q;\mu_\mathcal{L}): \; [\xi]_0+ [\xi_0]_0=[\xi']_0+ [\xi_0]_0$$
We denote the class of a cocycle $\xi$ in $K^0(Z,Q;\mu_\mathcal{L})$ by $[\xi]$.
\end{define}

\begin{prop}
\label{standardrep}
Any class $z\in Gr(Z,Q;\mu_\mathcal{L})$ can be written as 
\begin{equation}
\label{standardeq}
z=[\xi]_0-[\xi']_0,
\end{equation}
where $\xi'$ is degenerate.
\end{prop}

\begin{proof}
The set of classes of the form \eqref{standardeq} contains all classes of the form $[\xi]_0$ for any $\mu_\mathcal{L}$-relative $K$-theory cocycle $\xi$. The set of such classes is closed under sum. Hence it suffices to prove that the inverse of a $\mu_\mathcal{L}$-relative $K$-theory cocycle is of the form \eqref{standardeq}. This is clear as for any cocycle $\xi$, $\xi\oplus \xi^{op}$ is degenerate and $-[\xi]_0=[\xi^{op}]_0-[\xi\oplus \xi^{op}]_0$.

\end{proof}

\begin{remark}
\label{stabilizationremark}
A consequence of Proposition \ref{standardrep} is that two $\mu_\mathcal{L}$-relative $K$-theory cocycles $\xi_1$ and $\xi_2$ equals each other in $Gr(Z,Q;\mu_\mathcal{L})$ if and only if there is a degenerate cocycle $\xi'$ such that
$$\xi_1+\xi'\cong \xi_2+\xi'.$$
\end{remark}

We are now ready to relate the group $K^0(Z,Q;\mu_\mathcal{L})$ to the other $K$-theory groups. Recall the notation $GL_\infty(B)=\varinjlimÊGL_N(B)$ for any unital Banach algebra $B$ and $K_1(B):=\pi_0(GL_\infty(B))$. We define $\delta_0:GL_\infty(C(Q,B))\to Gr(Z,Q;\mu_\mathcal{L})$ by representing  $x\in GL_\infty(C(Q,B))$ by a bundle isomorphism $u:Q\times B^n\to Q\times B^n$ and setting 
$$\delta_0(x):=[Z\times B^n,Z\times B^n,0,0,u]_0.$$
It is clear that $\delta_0$ maps homotopic elements to homotopic cocycles, hence we can define $\delta:K^1(Q;B)\to K^0(Z,Q;\mu_\mathcal{L})$ as the mapping induced from $\delta_0$. The mapping $\rho:K^0(Z,Q;\mu_\mathcal{L})\to K^0(Z;B)\oplus K^0(Q)$ is defined by 
\begin{equation}
\label{rhoformula}
\rho(\mathcal{E}_B,\mathcal{E}_B',E_\field{C},E_\C',\alpha):=([\mathcal{E}_B]-[\mathcal{E}_B'])\oplus([E_\field{C}]-[E_\field{C}']).
\end{equation}
A priori, the formula \eqref{rhoformula} defines a mapping $Gr(Z,Q;\mu_\mathcal{L})\to K^0(Z;B)\oplus K^0(Q)$ that respects the homotopy equivalence as well as degenerate equivalence, which induces the mapping $\rho$.

\begin{lemma}
\label{relativekcycles}
The sequence
\begin{align}
\label{relativekcyclesexact}
K^1(Z;B)\oplus K^1(Q)\xrightarrow{\sigma}K^1(Q;B)\xrightarrow{\delta}&K^0(Z,Q;\mu_\mathcal{L})\xrightarrow{\rho} \\
\nonumber
\xrightarrow{\rho} \,&K^0(Z;B)\oplus K^0(Q)\xrightarrow{\sigma}K^0(Q;B)
\end{align}
is exact.
\end{lemma}

\begin{proof}
It is straight-forward to verify that the sequence is a complex. We will prove exactness from the end to beginning in the sequence. Assume that $x\oplus y\in K^0(Z;B)\oplus K^0(Q)$ is such that $x|Q=y\otimes g^*\mathcal{L}$ in $K^0(Q;B)$. We can write $x=[\mathcal{E}_B]-[Z\times B^n]$ and $y=[F_\field{C}]-[Q\times \field{C}^k]$. The fact that $x|Q=y\otimes g^*\mathcal{L}$ means that there exists an $n'\in \mathbbm{N}$ and an isomorphism
$$\alpha:(\mathcal{E}_B\oplus Z\times B^{n'})|_Q\oplus g^*\mathcal{L}^k\xrightarrow{\sim} F_\field{C}\otimes _\field{C}g^*\mathcal{L}\oplus B^{n+n'}.$$
It follows that $(\mathcal{E}_B\oplus Z\times B^n,Z\times B^{n+n'},F_\field{C},Q\times \C^k,\alpha)$ is a $\mu_\mathcal{L}$-relative $K$-theory cocycle and
$$\rho[\mathcal{E}_B\oplus Z\times B^n,Z\times B^{n+n'},F_\field{C},Q\times \C^k,\alpha]=x\oplus y.$$
So the sequence is exact at $K^0(Z;B)\oplus K^0(Q)$.

From Proposition \ref{degeneratestabilization} it follows that any $z\in K^0(Z,Q;\mu_\mathcal{L})$ can be written in the form
$$z=[\mathcal{E}_B,Z\times B^n,E_\C,Q\times \C^k,\alpha].$$
If $z\in \ker \rho$, it follows that there are $K, k'\in \mathbbm{N}$ such that $F_\field{C}\oplus Q\times \field{C}^{k'}\cong Q\times \field{C}^K$. Hence we can assume that 
$$z=[\mathcal{E}_B,Z\times B^n,Q\times \C^k,Q\times \C^k,\alpha].$$
Again, while $z\in \ker \rho$ there are $N,n'\in \mathbbm{N}$ and a unitary $u:\mathcal{E}_B|_Q\oplus Q\times B^{n''}\to Q\times B^N$. After stabilization we arrive at the desired identity $\delta[u]=z$.

For any $[u]\in K^1(Q;B)$ represented by a $u:Q\times B^n\to Q\times B^n$, Remark \ref{stabilizationremark} implies that $\delta[u]=0$ if and only if for two degenerate cycles 
$$(\mathcal{F}_{B,i}\oplus \mathcal{E}_{B,i}\oplus \mathcal{E}_{B,i}',\mathcal{F}_{B,i}'\oplus \mathcal{E}_{B,i}'\oplus \mathcal{E}_{B,i},F_{\C,i}\oplus E_{\field{C},i}\oplus E_{\C,i}',F_{\C,i}'\oplus E_{\field{C},i}'\oplus E_{\C,i},\alpha_i),\quad i=1,2,$$
where $\alpha_i=\alpha_{0,i}\oplus \tilde{\alpha}_i\oplus \tilde{\alpha}_i^{-1}$ is as in Definition \ref{defofdeg}, there is a homotopy
\begin{align*}
(B^n\oplus \mathcal{F}_{B,1}\oplus \mathcal{E}_{B,1}\oplus \mathcal{E}_{B,1}'&,B^n\oplus \mathcal{F}_{B,1}'\oplus \mathcal{E}_{B,1}'\oplus \mathcal{E}_{B,1},\\
&F_{\C,1}\oplus E_{\field{C},1}\oplus E_{\C,1}',F_{\C,1}'\oplus E_{\field{C},1}'\oplus E_{\C,1},u\oplus \alpha_1)\sim_h\\
(\mathcal{F}_{B,2}\oplus \mathcal{E}_{B,2}\oplus \mathcal{E}_{B,2}'&,\mathcal{F}_{B,2}'\oplus \mathcal{E}_{B,2}'\oplus \mathcal{E}_{B,2},\\
&F_{\C,2}\oplus E_{\field{C},2}\oplus E_{\C,2}',F_{\C,2}'\oplus E_{\field{C},2}'\oplus E_{\C,2},\alpha_2).
\end{align*}
By assumption, we can write $\alpha_{0,i}=\alpha_{0,i}^0|_Q\oplus \alpha_{0,i}'\otimes g^*\mathcal{L}$ for an isomorphism $\alpha_{0,i}^0$ of $B$-bundles on $Z$ and $\alpha_{0,i}'$ of vector bundles on $Q$ as in Definition \ref{defofdeg}. It follows that 
$$[u]=[\alpha_{0,2}^0]|_Q+ [\alpha_{0,2}']\otimes g^*\mathcal{L}-[\alpha_{0,1}^0]|_Q- [\alpha_{0,1}']\otimes g^*\mathcal{L}=\sigma\left([\alpha_{0,2}^0]-[\alpha_{0,1}^0], [\alpha_{0,2}']- [\alpha_{0,1}']\right).$$
\end{proof}

%

\begin{remark}
The motivation for not only considering easy cocycles is two-fold. The first reason is that our proof of exactness at $K^0(Z;B)\oplus K^0(Q)$ in Lemma \ref{relativekcycles} requires cocycles that are not necessarily easy. The second reason is that if $\mathcal{L}$ is non-trivializable, it is unclear how to define a mapping $\delta$ from $K^1(Q;B)$ into a $K$-theory group generated by easy and degenerate cocycles. 
\end{remark}

The $\mu_\mathcal{L}$-relative $K$-theory inherits a $K^*(Z)$-module structure from its components making the sequence in Lemma \ref{relativekcycles} $K^0(Z)$-linear. Let us describe this module structure on the $\mu_\mathcal{L}$-relative $K$-theory cocycles. 

\begin{prop}
The $K^0(Z)$-module structure on $K^0(Z,Q;\mu_\mathcal{L})$ is constructed as follows; if $V\to Z$ is a vector bundle and $(\mathcal{E}_B,F_\C,\alpha)$ a  $\mu_\mathcal{L}$-relative $K$-theory cocycle we define
$$(\mathcal{E}_B,\mathcal{E}_B',E_\field{C},E_\C',\alpha)\otimes V:=(\mathcal{E}_B\otimes V,\mathcal{E}_B'\otimes V,E_\field{C}\otimes V|_Q,E_\C'\otimes V|_Q,\alpha \otimes \mathrm{id}_{V|_Q}).$$
\end{prop}

\begin{proof}
It is clear that this tensor product, on the level of cocycles, makes the mappings in Lemma \ref{relativekcycles} $K^0(Z)$-linear. To verify that the tensor product induces a well defined map on $K$-theory, since tensor products with vector bundles preserve homotopy of cocycles, we only need to prove that if $\xi$ is stably degenerate, so is $\xi\otimes V$. If $(\mathcal{E}_B,\mathcal{E}_B',E_\field{C},E_\C',\alpha)$ is stably degenerate, it follows from the proof of Lemma \ref{relativekcycles} that the class of $(\mathcal{E}_B,\mathcal{E}_B',E_\field{C},E_\C',\alpha)\otimes V$ in $Gr(Z,Q;\mu)$ lies in the image of $\mathrm{im} \,\delta_0\circ \sigma$. More precisely, if $(\mathcal{E}_B,\mathcal{E}_B',E_\field{C},E_\C',\alpha)$ is stably degenerate then the class of $(\mathcal{E}_B,\mathcal{E}_B',E_\field{C},E_\C',\alpha)$ lies in $\mathrm{im}\, \delta_0\circ \sigma$ and both $\sigma$ and $\delta_0$ are $K^0(Z)$-linear. The exactness of the sequence \eqref{relativekcyclesexact} implies that $(\mathcal{E}_B,\mathcal{E}_B',E_\field{C},E_\C',\alpha)\otimes V$ is stably trivial.
\end{proof}

Recall that if $\pi:V\to Z$ is an even-dimensional spin$^c$ vector bundle the Bott class $\beta_V\in K^0(V)$, see \cite[Section 2.5]{Rav}, is the generator for the $K^*(Z)$-action on $K^*(V)$. Here $K^*(Z)$ act via pullback along $\pi$. The appropriate way for us to consider elements of $K^0(V)$ is as follows. We can consider $S(V\oplus 1_\field{R})$ as a fiber wise one point compactification. This is a sphere bundle $\pi_S:S(V\oplus 1_\field{R})\to Z$. There is a natural embedding $V\hookrightarrow S(V\oplus 1_\field{R})$ with dense and open image such that $S(V\oplus 1_\field{R})\setminus V=Z$. An element of $K^0(V)$ is an element of $K^0(S(V\oplus 1_\field{R}))$ vanishing on the complement of $V$. 

Whenever $\pi:V\to Z$ is a vector bundle, the two triples $(S(V\oplus 1_\field{R}),S(V|_Q\oplus 1_\field{R}), g\circ \pi_S|_{\pi^{-1}_S(Q)})$ and $(S(V\oplus 1_\field{R})\setminus V,S(V|_Q\oplus 1_\field{R})\setminus V|_Q, g\circ \pi_S|_{\pi^{-1}_S(Q)\setminus V|_Q})$ satisfies our standing assumption. We can define 
\begin{align*}
K^*(V,V|_Q;\mu_\mathcal{L}):=&\\
\ker (K^*(S(V\oplus 1_\field{R}),S(V|_Q\oplus 1_\field{R}), \mu_\mathcal{L})&\to K^*(S(V\oplus 1_\field{R})\setminus V,S(V|_Q\oplus 1_\field{R})\setminus V|_Q,\mu_\mathcal{L})).
\end{align*}

\begin{prop}[Thom isomorphism]
Let $(Z,Q,g)$, $\pi:V\to Z$ and $\beta_V\in K^0(V)$ be as above. The mapping
\[K^*(Z,Q;\mu_\mathcal{L})\to K^*(V,V|_Q;\mu_\mathcal{L}), \; \xi\mapsto \pi^*\xi\otimes \beta_V,\]
is an isomorphism.
\end{prop}

\begin{proof}
The mapping is well defined since 
$$\beta_V\in \ker(K^0(S(V\oplus 1_\field{R}))\to K^0(S(V\oplus 1_\field{R})\setminus V)).$$ 
The proposition follows from the fact that the sequence \eqref{KthExtSeq} is exact and the analogously defined mappings induce Thom isomorphisms $K^*(Q;B)\cong K^*(V|_Q;B)$ and $K^*(Z;B)\oplus K^*(Q)\cong K^*(V;B)\oplus K^*(V|_Q)$.
\end{proof}

\begin{prop}
\label{allexactnesseverwanted}
Let $Z$ be a compact Hausdorff space, $Q \subseteq Z$ a closed subspace and $g:Q\rightarrow X$ is a continuous map.  Then
\begin{enumerate}
\item Let $Z^{\prime}$ be another compact Hausdorff space, $Q^{\prime} \subseteq Z^{\prime}$ a closed subspace, and $g^{\prime}:Q^{\prime} \rightarrow X$ a continuous map.  If $g^{\prime}|_{Q\cap Q^{\prime}} = g|_{Q\cap Q^{\prime}}$, then the following sequence is exact:
\begin{align*}
K^0(Z\cup_{Z\cap Z^{\prime}} Z^{\prime}, Q\cup_{Q\cap Q^{\prime}} Q^{\prime};\mu_\mathcal{L})& \to\\
\to K^0(Z,Q;\mu_X)\oplus &K^0(Z^{\prime},Q^{\prime};\mu_\mathcal{L}) \to\\
&\to  K^0(Z\cap Z^{\prime},Q\cap Q^{\prime};\mu_\mathcal{L}),
\end{align*}
where the first mapping is defined by $\xi\mapsto \xi|_{(Z,Q)}\oplus \xi|_{(Z',Q')}$ and the second mapping by 
\[\xi^Z\oplus \xi^{Z'}\mapsto \xi^Z|_{(Z\cap Z^{\prime},Q\cap Q^{\prime})}- \xi^{Z'}|_{(Z\cap Z^{\prime},Q\cap Q^{\prime})}.\]
\item Let $Z^{\prime}$ be another compact Hausdorff space with $Q \subseteq Z^{\prime}$ a closed subspace, and $g^{\prime}:Z^{\prime} \rightarrow X$ a continuous map.  If $g^{\prime}|_{Q} = g$, then the following sequence is exact:
$$K^0(Z\cup_{Q} Z^{\prime}, Z^{\prime} ;\mu_\mathcal{L}) \rightarrow K^0(Z,Q;\mu_\mathcal{L})\oplus K^0(Z^{\prime}) \rightarrow K^0(Q), $$
where the first mapping is defined by $\xi\mapsto \xi|_{(Z,Q)}\oplus \xi_\C$ and the second mapping by $\xi^Z\oplus x\mapsto \xi^Z_\C-x|_Q$.
\end{enumerate}
\end{prop}

\begin{proof}
Follows from the exactness of the sequence \eqref{KthExtSeq}, the Mayer-Vietoris sequence for $K$-theory with coefficients and the nine lemma.
\end{proof}

\section{Geometric Cycles}
\label{sectiongeometriccycles}

In this section, $B$ denotes a unital Banach algebra, $X$ a locally compact Hausdorff space and $\mathcal{L}\to X$ a locally trivial bundle of finitely generated projective $B$-modules. Let us remind the reader that the example to keep in mind is the following one: $\Gamma$ is a discrete group acting freely and cocompactly on $\tilde{X}$, $B$ is a Banach algebra completion of $\C[\Gamma]$ and $\mathcal{L}:=\tilde{X}\times _\Gamma B$ is a flat $B$-bundle over $X:=\tilde{X}/\Gamma$.

\begin{define} \label{cycDef}
A geometric cycle relative to $\mathcal{L}$ is given by, $(W,(\mathcal{E}_{B},\mathcal{E}_{B}',E_{\field{C}},E_\C',\alpha),f)$, where 
\begin{enumerate}
\item $W$ is a smooth compact spin$^c$-manifold with boundary;
\item $\mathcal{E}_B$ and $\mathcal{E}_{B}'$ are locally trivial smooth finitely generated projective $B$-module bundles over $W$;
\item $E_{\field{C}}$ and $E_\C'$ are smooth Hermitean vector bundles over $\partial W$;
\item $f:\partial W \rightarrow X$ is a continuous map;
\item $\alpha:\mathcal{E}_B|_{\partial W} \oplus  (E_\C'\otimes f^*\mathcal{L}_X)\rightarrow \mathcal{E}_{B}'|_{\partial W}\oplus (E_{\field{C}} \otimes f^*\mathcal{L}_X)$ is a smooth isomorphism of $B$-module bundles.
\end{enumerate}
\end{define}

A number of remarks are in order; all are in keeping with the Baum-Douglas model for K-homology. The smoothness assumptions on the $B$-module bundles can be lifted, see \cite[Theorem $3.14$]{Sch}. The manifold in a cycle need not be connected.  As such, the bundles can have varying fiber dimensions.  A geometric cycle relative to $\mathcal{L}$, $(W,(\mathcal{E}_{B},\mathcal{E}_{B}',E_{\field{C}},E_\C',\alpha),f)$, is even (resp. odd) if the dimension of each of the connected components of $W$ is even (resp. odd) dimensional.  Condition $5$ of Definition \ref{cycDef} guarantees that $(\mathcal{E}_{B},\mathcal{E}_{B}',E_{\field{C}},E_\C',\alpha)$ defines a smooth cocycle for $K^0(W,\partial W,\mu_\mathcal{L})$. There is a natural definition of isomorphism of cycles and it is often the case that when we refer to a ``cycle" we are in fact refering to an ``isomorphism class of a cycle".  The addition operation on cycles is given by disjoint union (written as $\dot{\cup}$) and the opposite of cycle is obtained by taking the opposite spin$^c$ structure on the manifold in the cycle; taking the opposite is denoted by 
$$-(W,(\mathcal{E}_{B},\mathcal{E}_{B}',E_{\field{C}},E_\C',\alpha),f)=(-W,(\mathcal{E}_{B},\mathcal{E}_{B}',E_{\field{C}},E_\C',\alpha),f).$$

\begin{define}
A regular domain, $Y$, of a manifold $M$ is a closed submanifold of $M$ such that
\begin{enumerate}
\item ${\rm int}(Y)\ne \emptyset$;
\item If $p\in \partial Y$, then there exists a coordinate chart near $p$ in $M$, $\phi:U \rightarrow \field{R}^n$ centered at $p$ such that $\phi(Y \cap U)=\{ x \in \phi(U) \: | \: x_n\ge 0 \} $. 
\end{enumerate}
\end{define}

\begin{define} \label{borDef}
A bordism (relative to $\mathcal{L}$) is given by $$((Z,W),(\mathcal{E}_B,\mathcal{E}_{B}',E_{\field{C}},E_\C',\alpha),f)$$
where
\begin{enumerate}
\item $Z$ is a smooth compact spin$^c$-manifold with boundary;
\item $W$ is a smooth compact spin$^c$-manifold with boundary which is a regular domain of $\partial Z$;
\item $\mathcal{E}_B$ and $\mathcal{E}_{B}'$ are smooth $B$-module bundles over $Z$;
\item $E_{\field{C}}$ and $E_\C'$ are smooth Hermitean vector bundles over $\partial Z- {\rm int}(W)$;
\item $f:\partial Z -{\rm int}(W) \rightarrow X$ is a continuous map;
\item $\alpha:\mathcal{E}_B|_{\partial Z -{\rm int}(W)}  \oplus  (E_\C'\otimes f^*\mathcal{L}_X)\rightarrow \mathcal{E}_{B}'|_{\partial Z -{\rm int}(W)} \oplus (E_{\field{C}} \otimes f^*\mathcal{L}_X)$ is an isomorphism of $B$-module bundles.
\end{enumerate}
The boundary of such a bordism is given by 
\begin{equation}
\label{boundarycycleequation}
(W,(\mathcal{E}_B|_W,\mathcal{E}_{B}'|_W,E_{\field{C}}|_{\partial W},E_\C'|_{\partial W} ,\alpha|_{\partial W}),f|_{\partial W}).
\end{equation}
We say two cycles, $\epsilon$ and $\epsilon^{\prime}$ are bordant (we write $\epsilon \sim_{bor} \epsilon^{\prime}$) if $\epsilon \: \dot{\cup} -\epsilon^{\prime}$ is a boundary.
\end{define}

\begin{remark} \label{bordismRelationRemark}
The reader should note that if a cycle of the form \eqref{boundarycycleequation} is the boundary of a bordism $((Z,W),(\mathcal{E}_B,\mathcal{E}_{B}',E_{\field{C}},E_\C',\alpha),f)$, then Baum-Douglas cycles defined from the boundaries $(\partial W, E_{\field{C}}|_{\partial W}, f|_{\partial W})$ and $(\partial W, E_{\field{C}}'|_{\partial W}, f|_{\partial W})$ are boundaries in $K_*^{geo}(X)$.  If we use the notation of Definition \ref{borDef}, explicit bordisms (in $K_*^{geo}(X)$) can be formed by taking $(\partial Z - {\rm int}(W),E_{\field{C}},f)$ respectively $(\partial Z - {\rm int}(W),E^{\prime}_{\field{C}},f)$.
\end{remark}

\begin{define}
We say that a cycle $(W,\xi,f)$ relative to $\mathcal{L}$ is degenerate if $\xi$ is a degenerate cocycle (see Definition \ref{defofdeg}). 
\end{define}

\begin{prop}
\label{bordismprop}
Bordism is an equivalence relation. The set of equivalence classes of geometric cycles relative to $\mathcal{L}$ under bordism forms an abelian group and the set of equivalence classes of degenerate geometric cycles a subgroup thereof. 
\end{prop}

\begin{proof}
It is clear that bordism is an equivalence relation. That the set of equivalence classes form an abelian group follows from that for any geometric cycle $(W,(\mathcal{E}_B,\mathcal{E}_{B}',E_{\field{C}},E_\C',\alpha),f)$ the cycle 
$$(W,(\mathcal{E}_B,\mathcal{E}_{B}',E_{\field{C}},E_\C',\alpha),f)\dot{\cup}(-W,(\mathcal{E}_B,\mathcal{E}_{B}',E_{\field{C}},E_\C',\alpha),f)$$ 
is null bordant. The set of equivalence classes of degenerate geometric cycles is closed under disjoint union. If $(W,\xi,f)$ is degenerate, so is $(-W,\xi,f)$ and the bordism classes of degenerate cycles is closed under inverse. 
\end{proof}

\begin{prop}
\label{homotopypropforcycles}
Let $(W,\xi,f)$ and $(W,\xi',f)$ be geometric cycles relative to $\mathcal{L}$. If $\xi$ and $\xi'$ are homotopic, then there is a bordism 
$$(W,\xi,f)\sim_{bor}(W,\xi',f).$$
\end{prop}

\begin{proof}
By Proposition \ref{homotopyprop}, there is a smooth cocycle $\tilde{\xi}$ relative to $\mathcal{L}$ over $(W\times [0,1],\partial W\times [0,1],f\circ \pi)$ isomorphic to $\xi$ and $\xi'$ in respective end points. Here $\pi:\partial W\times [0,1]\to \partial W$ denotes the projection. After "straightening the angle" (see \cite{Rav}), we can assume that $Z=W\times [0,1]$ is a smooth manifold with boundary $-W\cup_{\partial W} \partial W\times [0,1]\cup_{\partial W} W$. It follows that $((Z,-W\dot{\cup}W), \tilde{\xi},f\circ \pi)$ forms a bordism whose boundary is $(-W,\xi,f)\dot{\cup}(W,\xi',f)$.
\end{proof}

If $W$ is a smooth compact spin$^c$-manifold with boundary, $f:\partial W\to X$ a continuous mapping and $[\xi]\in K^0(W,\partial W;\mu_\mathcal{L})$; we call the triple $(W,[\xi],f)$ a cycle relative to $\mathcal{L}$ with $K$-theory data. The notion of bordism as in Definition \ref{borDef} adapts (in a very natural way) to cycles with $K$-theory data. 

\begin{prop}
\label{ktheoryincycles}
The abelian group of bordism classes of geometric cycles relative to $\mathcal{L}$ modulo the group of degenerate cycles is isomorphic to the abelian group of bordism classes of geometric cycles relative to $\mathcal{L}$ with $K$-theory data via the mapping 
\begin{equation}
\label{bundlestoclasses}
(W,\xi,f)\mapsto (W,[\xi],f).
\end{equation}
\end{prop}

\begin{proof}
Since the mapping in Equation \eqref{bundlestoclasses} takes bordisms to bordisms it produces a well defined mapping on bordism classes. The proof is complete if we can produce an inverse mapping. If we have a cycle with $K$-theory data $(W,[\xi],f)$, and represent $[\xi]=[\xi_1]-[\xi_2]$ were $\xi_1=(\mathcal{E}_1,\mathcal{E}_1',E_1,E_1',\alpha_1)$ and $\xi_2=(\mathcal{E}_2,\mathcal{E}_2',E_2,E_2',\alpha_2)$ are relative $K$-theory cocycles. We claim that, modulo degenerate cycles, the bordism class of the cycle 
$$(W,\xi_1,f)\dot{\cup}(-W,\xi_2,f)$$ 
only depends on $(W,[\xi],f)$. By Proposition \ref{homotopypropforcycles}, homotopic cocycles define bordant cycles. If $[\xi]=[\tilde{\mathcal{E}}_1,\tilde{\mathcal{E}}_1',\tilde{E}_1,\tilde{E}_1',\tilde{\alpha}_1]-[\tilde{\mathcal{E}}_2,\tilde{\mathcal{E}}_2',\tilde{E}_2,\tilde{E}_2',\tilde{\alpha}_2]$, there is a bundle $(\hat{\mathcal{E}}_1,\hat{\mathcal{E}}_1',\hat{E}_1,\hat{E}_1',\hat{\alpha})$ such that 
\begin{align*}
(W,(\mathcal{E}_1,\mathcal{E}_1',E_1,E_1',\alpha_1),f)\dot{\cup}(W,(\tilde{\mathcal{E}}_2,\tilde{\mathcal{E}}_2',\tilde{E}_2,\tilde{E}_2',\tilde{\alpha}_2),f)\dot{\cup}(W,(\hat{\mathcal{E}}_1,\hat{\mathcal{E}}_1',\hat{E}_1,\hat{E}_1',\hat{\alpha}),f)=\\
(W,(\tilde{\mathcal{E}}_1,\tilde{\mathcal{E}}_1',\tilde{E}_1,\tilde{E}_1',\tilde{\alpha}_1),f)\dot{\cup}(W,(\mathcal{E}_2,\mathcal{E}_2',E_2,E_2',\alpha_2),f)\dot{\cup}(W,(\hat{\mathcal{E}}_1,\hat{\mathcal{E}}_1',\hat{E}_1,\hat{E}_1',\hat{\alpha}),f).
\end{align*}
After using the group property of the bordism classes, the third term cancels giving 
\begin{align*}
(W,(\mathcal{E}_1,\mathcal{E}_1',E_1,E_1',\alpha_1),f)&\dot{\cup}(W,(\tilde{\mathcal{E}}_2,\tilde{\mathcal{E}}_2',\tilde{E}_2,\tilde{E}_2',\tilde{\alpha}_2),f)\sim_{bor} \\
&(W,(\tilde{\mathcal{E}}_1,\tilde{\mathcal{E}}_1',\tilde{E}_1,\tilde{E}_1',\tilde{\alpha}_1),f)\dot{\cup}(W,(\mathcal{E}_2,\mathcal{E}_2',E_2,E_2',\alpha_2),f).
\end{align*}
It ensures that the bordism class of the cycle $(W,\xi_1,f)\dot{\cup}(-W,\xi_2,f)$ modulo degenerate cycles does not depend on the pre image of $[\xi]$ in $Gr(W,\partial W;\mu)$ since equivalent $\mu$-relative $K$-theory cocycles produce cycles bordant to degenerate cycles. 
\end{proof}

\begin{define}
Let $(W,[\xi],f)$ be a cycle with $K$-theory data and $V$ a smooth spin$^c$-vector bundle over $W$ with even rank.  The vector bundle modification of $(W,[\xi],f)$ by $V$ is defined to be the cycle:
$$(W^V,\pi^*[\xi]\otimes [\beta_V] ,f\circ \pi) $$
where
\begin{enumerate}
\item ${\bf 1}$ is the trivial real line bundle over $W$ (i.e., $W\times \field{R}$).
\item $W^V:=S(V\oplus {\bf 1})$ (i.e., the sphere bundle of $E\oplus {\bf 1}$).
\item $[\beta_V]\in K^0(V)\subseteq K^0(W^V)$ is the Bott element (see \cite[Section 2.5]{Rav});
\item $\pi:W^V \rightarrow W$ is the bundle projection.
\end{enumerate}
The vector bundle modification of a cycle $(W,[\xi],f)$ by $V$ is denoted by $(W,[\xi],f)^V$.
\end{define}

\begin{define} \label{equCyc}
Let $\mathcal{S}^{geo}_*(X;\mathcal{L})$ be the $\field{Z}/2$-graded abelian group given by (isomorphism classes of) cycles with $K$-theory data modulo the equivalence relation generated by
\begin{enumerate}
\item {\bf Bordism:} Bordant cycles are declared to be equivalent.
\item {\bf Vector bundle modification:}  Given a cycle, $(W,[\xi],f)$ and spin$^c$-vector bundle $V\to W$ of even rank, we declare that 
$$(W,[\xi],f) \sim (W,[\xi],f)^V $$
\end{enumerate}
If $X=B\Gamma$ and $\mathcal{L}=\mathcal{L}_{\mathcal{A}(\Gamma)}$ for a discrete finitely generated group $\Gamma$ and a Banach algebra completion $\mathcal{A}(\Gamma)$ of $\C[\Gamma]$ we use the notation 
$$\mathcal{S}^{geo}_*(\Gamma;\mathcal{A}):=\mathcal{S}^{geo}_*(B\Gamma;\mathcal{L}_{\mathcal{A}(\Gamma)})$$
\end{define}

It is more or less obvious that $\mathcal{S}^{geo}_*(X;\mathcal{L})$ is an abelian semigroup.  That $\mathcal{S}^{geo}_*(X;\mathcal{L})$ is an abelian group follows from Proposition \ref{bordismprop}.  The $\field{Z}/2$-grading is given by the dimensionality of the connected components modulo two. The next proposition implies that addition within $\mathcal{S}^{geo}_*(X;\mathcal{L})$ is compatible with addition within $K^0(W, \partial W;\mu_{\mathcal{L}})$. The proof (which is left to the reader) is similar to the proof in \cite[Proposition 4.11]{DeeRZ} (also see \cite[Proposition 4.3.2]{Rav}).

\begin{prop}Let $(W,[\xi_1],f)$ and $(W,[\xi_2],f)$ be cycles. Then
$$(W,[\xi_1],f) \dot{\cup} (W,[\xi_2],f) \sim (W,[\xi_1] + [\xi_2],f).$$ 
\end{prop}

\begin{remark}
We could equally well vector bundle modify cycles with vector bundle data using the Bott bundle (see \cite[Section 10]{BD}). Since degenerate cycles are closed, and homotopy of cocycles is preserved, under the action of $K$-theory, the result will modulo degenerate cycles only differ by a boundary. The abelian group $\mathcal{S}^{geo}_*(X;\mathcal{L})$ can be obtained from cycles with vector bundle data as in Definition \ref{equCyc} after adding two more relations:
\begin{enumerate}
\item[(3)] {\bf Degenerates are trivial:} $(W,\xi,f)\sim 0$ if $\xi$ is degenerate.
\item[(4)] {\bf Direct sum/disjoint union:} Given cycles $(W,\xi_1,f)$ and $(W,\xi_2,f)$, we declare that
$$(W,\xi_1,f) \dot{\cup} (W,\xi_2,f) \sim (W,\xi_1\oplus \xi_2,f).$$ 
\end{enumerate}
\end{remark}

\section{Relation with the assembly map}

The main result of this section is the proof that the geometric surgery group $\mathcal{S}_*^{geo}(X;\mathcal{L})$ fits in a six term exact sequence with the assembly mapping $\mu_\mathcal{L}:K_*^{geo}(X)\to K_*^{geo}(pt;B)$. As a first iteration we prove that the six term sequence is a complex (Proposition \ref{sixTerExtSeqPart1}) and after adapting the notions of normal bordism to the geometric group  $\mathcal{S}_*^{geo}(X;\mathcal{L})$ in Subsection \ref{subsectionnormalbordism}, we prove that the six term sequence is exact (Theorem \ref{sixTerExtSeq}). The reader should note that the development here is similar to \cite[Section 4]{DeeRZ}.

\begin{prop} \label{sixTerExtSeqPart1}
Let $X$ be locally compact and $\mathcal{L}\to X$ a bundle of projective finitely generated $B$-modules. If we use the following notation
\begin{enumerate}
\item $\mu_\mathcal{L}:K_*^{geo}(X)\to K_*^{geo}(pt;B)$ is the assembly mapping;
\item $r:K_*^{geo}(pt;B) \rightarrow \mathcal{S}^{geo}_*(X;\mathcal{L})$ is defined at the level of cycles via $(M,\mathcal{E}_{B})$ is mapped to $(M,(\mathcal{E}_{B},M\times 0,\emptyset,\emptyset,\emptyset),\emptyset)$;
\item $\delta: \mathcal{S}^{geo}_*(X;\mathcal{L}) \rightarrow K_{*-1}^{geo}(X)$ is defined at the level of cycles via 
$$(W,(\mathcal{E}_{B},\mathcal{E}_{B}',E_{\field{C}},E_\C',\alpha),f)\mapsto (\partial W, E_{\field{C}},f)\dot{\cup} (-\partial W, E_\C',f);$$
\end{enumerate}
then $r \circ \mu_\mathcal{L}=0$, $\delta \circ r=0$ and $\mu_\mathcal{L}\circ \delta=0$.
\end{prop}

\begin{proof}
That the maps are well-defined follows from the compatiblity of the relations in Definition \ref{equCyc} and the relation of the Baum-Douglas model.  For example, for the map $\delta$ and the bordism relation see Remark \ref{bordismRelationRemark}.  The details for the rest of the relations are left to the reader. 
\par 
The bordisms relations in the various geometric models implies that the composition of successive map is zero.  The details for $r\circ \mu_\mathcal{L}=0$ are as follows.  Suppose that $(M,[E_{\field{C}}]-[E'_\C],f)$ is a cycle in $K_*^{geo}(X)$ with $K$-theory data.  We are required to show that 
$$(M,(E_{\field{C}}\otimes_{\field{C}}\mathcal{L},E_{\field{C}}'\otimes_{\field{C}}\mathcal{L},\emptyset,\emptyset,\emptyset),\emptyset)$$ 
is a boundary in $\mathcal{S}^{geo}_*(X;\mathcal{L})$.  This follows from the following bordism
$$((M\times [0,1], M\times \{1\}),(\pi^*(E_{\field{C}}\otimes_{\field{C}}\mathcal{L}),\pi^*(E_{\field{C}}\otimes_{\field{C}}\mathcal{L}),E_{\field{C}}\times \{0\},E_\C'\times\{0\},\mathrm{id}),f) $$
where $\pi:M\times [0,1] \rightarrow M$ is the projection map. 
\par
The composition $\delta \circ r$ is clearly zero since cycles in the image of $r$ have empty boundary.  
\par
Finally, we consider the case of $\mu_\mathcal{L} \circ \delta$.  As such, let $(W,(\mathcal{E}_{B},\mathcal{E}_B',E_{\field{C}},E_\C',\alpha),f)$ be a cycle in $\mathcal{S}^{geo}_*(X;\mathcal{L})$.  Then 
\begin{eqnarray*}
(\mu_\mathcal{L} \circ \delta)(W,(\mathcal{E}_{B},\mathcal{E}_B',E_{\field{C}},E_\C',\alpha),f) & = & \mu_\mathcal{L}(\partial W,E_{\field{C}},f)-\mu_\mathcal{L}(\partial W,E_{\field{C}}',f) \\
& = & (\partial W, E_{\field{C}}\otimes f^*(\mathcal{L}))-(\partial W, E_{\field{C}'}\otimes f^*(\mathcal{L})) \\
& = &(\partial W, [E_{\field{C}}\otimes f^*(\mathcal{L})]-[E_{\field{C}}'\otimes f^*(\mathcal{L})])\\
& = & \partial (W,[\mathcal{E}_B]-[\mathcal{E}_B'])=0
\end{eqnarray*}
This completes the proof.
\end{proof}

\subsection{Normal Bordism}
\label{subsectionnormalbordism}
\begin{define}
Let $E$ be a vector bundle.  Then, a vector bundle, $F$, is a complementary bundle for $E$, if $E\oplus F$ is a trivial vector bundle.  If $M$ is a manifold (possibly with boundary), then a complementary bundle for $TM$ will be called a normal bundle. 
\end{define}
Given a manifold with boundary $(W,\partial W)$, we can produce a normal bundle by taking a neat embedding in $H^n:=\{ (v_1,\ldots,v_n) \in \field{R}^n \: | \: v_1 \geq 0 \}$ for $n$ sufficiently large.  
\begin{define}
Let $(W,[\xi],f)$ and $(W^{\prime},[\xi^{\prime}],f^{\prime})$ be two cycles in $\mathcal{S}^{geo}_*(X;\mathcal{L})$.  Then, we say there is a normal bordism from $(W,[\xi],f)$ to $(W^{\prime},[\xi^{\prime}],f^{\prime})$ if there exists normal bundles $N_W$, $N_{W^{\prime}}$ (over $W$ and $W^{\prime}$ respectively) such that 
$$(W,[\xi],f)^{N_W} \sim_{bor} (W^{\prime},[\xi^{\prime}],f^{\prime})^{N_{W^{\prime}}}$$
This relation is denoted by $\sim_{nor}$.
\end{define}

The next lemma is a standard result in geometric topology (see for example \cite[Lemma 4.5.7]{Rav}); the lemma that follows it is a natural generalization of \cite[Lemma 4.4.3]{Rav}. We prove the latter for the convenience of the reader. 

\begin{lemma} \label{norStaIso}
Normal bundles are stably isomorphic (i.e., if $N_1$ and $N_2$ are normal bundles for $W$, then there exists trivial bundles, $E_1$ and $E_2$, such that $N_1 \oplus E_1 \cong N_1 \oplus E_2$).
\end{lemma}

\begin{lemma} \label{twoVecMod}
If $E_1$ and $E_2$ are spin$^c$-vector bundles with even dimensional fibers over a compact spin$^c$-manifold $W$ and $(W,\xi,f)$ is a cycle, then
\begin{equation*}
(W,[\xi],f)^{E_1 \oplus E_2} \sim_{bor} ((W,[\xi],f)^{E_1})^{p^*(E_2)}
\end{equation*}
where $p:S(E_1\oplus {\bf 1}) \rightarrow W$ is the projection map.
\end{lemma}

\begin{proof}
We can assume that $W$ is connected, so the ranks of $E_1$ and $E_2$ are constant; let us suppose that they are $2n_1$ and $2n_2$  respectively. We define $G_1:=Spin^c(2n_1)$, $G_2:=Spin^c(2n_2)$ and $G:=G_1\times G_2$. We will start by considering the case $W=pt$. Following the proof of \cite[Lemma $4.4.2$]{Rav}, we can construct a $G$-equivariant smooth embedding $S^{2n_1}\times S^{2n_2}\hookrightarrow \mathbbm{R}^{2n_1+2n_2+1}$ whose image lies inside an annulus $B_R\setminus B_\epsilon \subseteq \mathbbm{R}^{2n_1+2n_2+1}$. We let $W_0$ be the domain lying between the image of $S^{2n_1}\times S^{2n_2}$ and $S^{2(n_1+n_2)}=\partial B_R$. Define $\beta\in K^0_G(W_0)$ as the restriction of the Bott class in $K^0_G(B_R\setminus B_\epsilon)$ to $W_0$.  A straight-forward computation shows that $\beta|_{S^{2n_1}\times S^{2n_2}}=\beta_{2n_1}\otimes_\C\beta_{2n_2}\in K^0_G(S^{2n_1}\times S^{2n_2})$, the external product of the Bott classes, and $\beta|_{S^{2(n_1+n_2)}}=\beta_{2(n_1+n_2)}\in K^0_G(S^{2(n_1+n_2)})$-- the Bott class on $S^{2(n_1+n_2)}$. In particular, $(W_0,\beta)$ is a $G$-equivariant bordism $(S^{2n_1}\times S^{2n_2},\beta_{2n_1}\otimes_\C \beta_{2n_2})\sim_{bor}(S^{2(n_1+n_2)},\beta_{2(n_1+n_2)})$.

Returning to the general case, let $P_1\to W$ and $P_2\to W$ denote the spin$^c$-frame bundles of $E_1$ respective $E_2$ (i.e. $P_i\to W$ satisfies $S(E_i\oplus 1)=P_i\times_{G_i}S^{2n_i}$). We define the $G$-bundle $P:=P_1\times_WP_2\to W$ and
\[Z:=P\times_G W_0.\]
Using the ``straightening the angle"-technique (see \cite[Lemma $4.1.9$]{Rav}), we can assume that $Z$ is a smooth manifold with boundary. We note that 
$$\partial Z=\partial P\times_GW_0\cup_{\partial P\times_G \partial W_0} P\times_G \partial W_0.$$ 
Furthermore,
$$P\times_G \partial W_0=P\times_G S^{2(n_1+n_2)}\dot{\cup} (-P\times_G (S^{2n_1}\times S^{n_2}))=W^{E_1 \oplus E_2}\dot{\cup}(W^{E_1})^{p^*(E_2)}.$$
If we let $\pi:Z\to W$ and $\pi_0:Z\to W_0$ denote the projections, $\pi^*\xi\otimes \pi_0^*\beta\in K^0(Z,\partial P\times_GW_0;\mu_\mathcal{L})$ is well defined and
\begin{align*}
\partial((Z,W^{E_1 \oplus E_2}\dot{\cup}(W^{E_1})^{p^*(E_2)}),&[\pi^*\xi\otimes \pi_0^*\beta],f\circ \pi|_{\partial P\times_GW_0})=\\
&=(W,[\xi],f)^{E_1 \oplus E_2}\dot{\cup}((-W,[\xi],f)^{E_1})^{p^*(E_2)}.
\end{align*}
\end{proof}

The proof of the next proposition is similar to the proof of \cite[Proposition 4.17]{DeeRZ} (also see \cite[Section 4.5]{Rav}); the proof is left to the reader.

\begin{prop}
The relation of normal bordism of cycles is an equivalence relation.  Moreover, it is equal to the relation constructed in Definition \ref{equCyc}. 
\end{prop}

\begin{cor} 
\label{norBorToTrivial}
A cycle (of the form $(W,[\xi],f)$) in $\mathcal{S}^{geo}_*(X;\mathcal{L})$ is trivial if and only if it normally bounds.
\end{cor}

\subsection{Six-term exact sequence}

\begin{theorem} 
\label{sixTerExtSeq}
Assume that $X$ is a locally compact Hausdorff space, $B$ is a unital Banach algebra and $\mathcal{L}\to X$ is a locally trivial bundle of finitely generated projective $B$-modules. The following sequence is exact
\begin{equation}
\label{thesequence}
\begin{CD}
K_0^{geo}(X) @>\mu_\mathcal{L}>> K_0^{geo}(pt;B) @>r>>  \mathcal{S}^{geo}_0(X;\mathcal{L})\\
@AA\delta A @. @VV\delta V \\
 \mathcal{S}^{geo}_1(X;\mathcal{L}) @<r<<  K_1^{geo}(pt;B) @<\mu_\mathcal{L}<< K_1^{geo}(X) 
\end{CD}
\end{equation}
where $\mu_\mathcal{L}$, $r$ and $\delta$ is as in Proposition \ref{sixTerExtSeqPart1}.
\end{theorem}

The reader should recall the basic properties of $K^0(W,\partial W;\mu_\mathcal{L})$ discussed in Lemma \ref{relativekcycles} and Equation \eqref{KthExtSeq}. That the mappings are well defined and that the sequence \eqref{thesequence} is a complex follows from Proposition \ref{sixTerExtSeqPart1}. The proof will be completed upon showing 
$${\rm ker}(\mu_\mathcal{L}) \subseteq {\rm im}(\delta),\:  {\rm ker}(\delta)  \subseteq {\rm im}(r),\: {\rm ker}(r) \subseteq  {\rm im}(\mu_\mathcal{L}). $$
It is here that the notion of normal bordism is required. The proof is similar to the proof of \cite[Theorem 2.20]{Dee1}, \cite[Theorem 4.19]{DeeRZ}; it uses ideas from \cite{Jak} and \cite{Rav}.

\begin{proof}[Proof of Theorem \ref{sixTerExtSeq}]
To prove $ {\rm ker}(\delta)  \subseteq {\rm im}(r)$, let $(W,[\xi],f)$ be a cycle such that $(\partial W, [\xi_{\field{C}}],f)$ is trivial in $K_*^{geo}(X)$.  We begin by showing that, without loss of generality, $(\partial W, [\xi_{\field{C}}],f)$ can be assumed to be a boundary.  By \cite[Corollary 4.5.16]{Rav}, there exists normal bundle over $\partial W$ such that $(\partial W, [\xi_{\field{C}}],f)^N=\partial (Q,[\eta],g)$.  The normal bundle, $N$, may not extend to all of $W$.  However, using Lemma \ref{norStaIso}, we have that if $N_W$ is a normal bundle for $W$, then there exists trivial bundle $\epsilon$ over $W$ such that    
\begin{eqnarray*}
\partial (W,[\xi],f)^{N_W} & = & (\partial W, [\xi_{\field{C}}],f)^{N_W|_{\partial W}} \\
& = & (\partial W, [\xi_{\field{C}}],f)^{N\oplus \epsilon} \\
& \sim_{bor} & ((\partial W, [\xi_{\field{C}}],f)^N)^{p^*(\epsilon)} 
\end{eqnarray*}
This completes the first step in the proof; that is, without loss of generality, we can assume that given $(W,[\xi],f)$ in the kernel of $\delta$, $\delta(W,[\xi],f)=(\partial W, [\xi_{\field{C}}],f)$ is a boundary. 
\par
Let $(Q,[\eta],g)$ denote a bordism (in $K_*^{geo}(X)$) with boundary $(\partial W, [\xi_{\field{C}}],f)$.  Form a cycle in $K_*^{geo}(pt;B)$, by taking
$$(Z,[\tilde{\xi}_B]):= (W,[\xi_{B}]) \cup_{\partial W} (Q,[\eta \otimes g^*(\mathcal{L})])$$
The proof is completed by showing $r(Z,[\tilde{\xi}_B]) \sim (W,[\xi],f)$.  This fact follows from the following bordism
$$((Z\times [0,1], Z\times \{0\} \dot{\cup} W\times \{1\}),[\hat{\xi}],g)$$
where only the K-theory data, $[\hat{\xi}]$, requires additional explanation.  It is constructed as follows.  Let $\pi:Z\times [0,1] \rightarrow Z$ be the projection map.  The class $([\pi^*(\tilde{\xi}_B)],[\eta]) \in K^0(Z\times [0,1];B)\oplus K^0(Q)$ is in the kernel of the map $K^0(Z\times [0,1];B)\oplus K^0(Q) \rightarrow K^0(Q;B)$.  Exactness implies that this class has a preimage in $K^0(Z\times [0,1],Q;g^*\mu)$ and let $[\hat{\xi}]$ be any such preimage, see Proposition \ref{allexactnesseverwanted}.  
\par
The boundary of the bordism, $((Z\times [0,1], Z \dot{\cup} W),[\hat{\xi}],g)$ is given by 
$$(-Z,[\hat{\xi}|_{Z}],\emptyset)\dot{\cup}(W,[\hat{\xi}|_{W}],g|_{\partial W})=r(-(Z,[\tilde{\xi}_B]))\dot{\cup}(W,[\xi],f)$$
as required. 
\par
The case ${\rm ker}(\mu_\mathcal{L}) \subseteq {\rm im}(\delta)$ is considered next.  Let $(Q,[\eta],g)$ be a cycle in $K_*^{geo}(X)$ such that $(Q,[\eta\otimes g^*(\mathcal{L})])$ is trivial in $K_*^{geo}(pt;B)$.  There exists a normal bundle $N$ over $Q$ such that $(Q,[\eta\otimes g^*(\mathcal{L})])^N$ is a boundary.  Let $(W,[\xi_B])$ denote such a bordism.  Form a cycle in $\mathcal{S}^{geo}_*(X;\mathcal{L})$ by taking $(W,[\xi],g\circ \pi_N)$ where $\pi_N$ is the projection coming from the vector bundle modification by $N$ and the class $[\xi]$ is a preimage of the class $([\xi_B],[\eta\otimes \beta_N])$ in the exact sequence \eqref{KthExtSeq}.  The reader should note that $([\xi_B],[\eta])$ is by construction in the kernel of the map $K^0(W;B)\oplus K^0(Q)$.  Finally, 
$$\delta(W,[\xi],g)=(\partial W, [\eta\otimes \beta_N],g\circ \pi_N)=(Q,[\eta],g)^N\sim (Q,[\eta],g)$$
as required. 
\par
Finally, we show that ${\rm ker}(r) \subseteq  {\rm im}(\mu_\mathcal{L})$.  Let $(M,[\xi_B])$ be a cycle in $K_*^{geo}(pt;B)$ such that $r(M,[\xi_B])$ is trivial in $\mathcal{S}^{geo}_*(X;\mathcal{L})$.  By Corollary \ref{norBorToTrivial}, there exists a normal bundle, $N$, and a bordism in $\mathcal{S}^{geo}_*(X;\mathcal{L})$, $((W,M^N),[\tilde{\xi}],f)$, such that 
$$\partial ((W,M^N),[\tilde{\xi}],f)=(M,[\xi_B])^N.$$ 
Since $\partial M^N=\emptyset$, $\partial W - M^N$ also has empty boundary.  Hence $(\partial W - M^N,[\tilde{\xi}_{\field{C}}],f)$ is a cycle in $K_*^{geo}(X)$.  Moreover, $(W,[\tilde{\xi}_B])$ is a bordism in $K_*^{geo}(pt;B)$ between $\mu_X(\partial W - M^N,[\tilde{\xi}_{\field{C}}],f)$ and $(M,\xi_B)^N$.  This completes the proof.
\end{proof}

\begin{remark}
\label{orientedmodel}
We note that there is a version of $K$-homology modeled on orientable manifolds (see \cite{Kescont} or \cite{HReta}). A cycle in this theory for $X$ consists of $(M,S,g)$ where $M$ is an oriented closed manifold, $S\to M$ a complex Clifford bundle and $g:M\to X$ is continuous. The relations are analogous to the ones above, however vector bundle modification is carried out using oriented vector bundles. That the two models are isomorphic is proven in Section $2.3$ of \cite{Kescont}. As a consequence of Theorem \ref{sixTerExtSeq}, we can also model $\mathcal{S}_*^{geo}(X;\mathcal{L})$ using oriented manifolds. To be precise, a cycle in the oriented version is given by $(W,(\mathfrak{S}_B,\mathfrak{S}'_B,S_\C,S'_\C,\alpha),g)$ where $W$ is an oriented manifold with boundary, $\mathfrak{S}_B$ and $\mathfrak{S}'_B$ are Clifford $B$-modules on $W$, $S_\C$ and $S_\C'$ are Clifford modules on $\partial W$ and $g:\partial W\to X$ is a continuous mapping. The isomorphism $\alpha$ should form an isomorphism of certain Clifford $B$-modules on $\partial W$. 
\end{remark}

\section{Products}
Let $\Gamma_1$ and $\Gamma_2$ be two finitely generated discrete groups. In \cite{Seg}, a product $K_n(B\Gamma_1) \times \mathcal{S}^{ana}_{n^{\prime}}(\Gamma_2) \rightarrow \mathcal{S}^{ana}(\Gamma_1 \times \Gamma_2)$ is defined. We define the corresponding product in the context of our geometric model. For simplicity, we assume that both $B\Gamma_1$ and $B\Gamma_2$ are finite CW-complexes; the general situation follows from this special case using an inductive limit argument.
\par
Assume that $\mathcal{A}'(\Gamma_1)$, $\mathcal{A}''(\Gamma_2)$ and $\mathcal{A}(\Gamma_1\times \Gamma_2)$ are Banach algebra completions of $\C[\Gamma_1]$, $\C[\Gamma_2]$ respectively $\C[\Gamma_1\times \Gamma_2]$. We will make the assumption that 
\begin{equation}
\label{inclusioncondition}
\mathcal{A}'(\Gamma_1)\otimes^{alg} \mathcal{A}''(\Gamma_2)\subseteq \mathcal{A}(\Gamma_1\times \Gamma_2).
\end{equation}
This condition is satisfied for instance when $\mathcal{A}=\mathcal{A}'=\mathcal{A}''=\ell^1$ or $\mathcal{A}=\mathcal{A}'=\mathcal{A}''=C^*_{\bf full}$ or $C^*_{\bf red}$. Let $(M,F,g)$ be a cycle in $K_n(B\Gamma_1)$ and $(W,[\xi] ,f)$ be a cycle in $\mathcal{S}^{geo}_{n^{\prime}}(\Gamma_2,\mathcal{A}'')$. The reader can check that \eqref{inclusioncondition} guarantees that the mapping 
$$\Psi : K^0(B\Gamma_1) \times K^0(W,\partial W;\mu_{\mathcal{L}_{\mathcal{A}''(\Gamma_2)}}) \rightarrow K^0(M\times W, M\times \partial W; \mu_{\mathcal{L}_{\mathcal{A}(\Gamma_1\times \Gamma_2)}})$$ 
defined at the level of cocycles via 
$$(F, (\mathcal{E}, \mathcal{E}^{\prime},E,E^{\prime},\alpha)) \mapsto (F\otimes g^*(\mathcal{L}_{\mathcal{A}'(\Gamma_1)}) \boxtimes \mathcal{E}, F^{\prime}\otimes g^*(\mathcal{L}_{\mathcal{A}'(\Gamma_1)}) \boxtimes \mathcal{E}^{\prime}, F\boxtimes E, F\boxtimes E^{\prime}, \mathrm{id} \boxtimes \alpha)$$
is well-defined. We note that $\boxtimes$ is defined by taking the (exterior) tensor product of the pullbacks of bundles under the obvious projection maps.
\begin{define}
\label{pairingdefinition}
Let $(M,F,g)$ be a cycle in $K_n(B\Gamma_1)$ and $(W,[\xi],f)$ be a cycle in $\mathcal{S}^{geo}_{n^{\prime}}(\Gamma_2,\mathcal{A}'')$ and form
$$(M\times W, \Psi(F, [\xi]) ,f\times g)$$
We denote the induced pairing by 
\begin{equation}
\label{khomsurgpairing}
< \cdot, \cdot >:K_n(B\Gamma_1) \times \mathcal{S}^{geo}_{n^{\prime}}(\Gamma_2,\mathcal{A}'') \rightarrow \mathcal{S}^{geo}_{n+n'}(\Gamma_1 \times \Gamma_2,\mathcal{A}).
\end{equation}
\end{define}

\begin{prop}
\label{productprop}
The pairing \eqref{khomsurgpairing} is well-defined.
\end{prop}

\begin{proof}
It is clear that the pairing respects the disjoint union relation. For the bordism relation in $K_*^{geo}(B\Gamma_1)$, we suppose $(M,F,g)$ is the boundary of $(\tilde{M},\tilde{F},\tilde{g})$. Then, by applying the ``straightening the angle" technique to
$$(\tilde{M}\times W, \Psi(\tilde{F}, [\xi]),\tilde{g}\times f ),$$
we can produce a bordism which has boundary given by $<(M,F,g),(W,[\xi],f)>$. 
\par
For the bordism relation in $\mathcal{S}^{geo}_*(\Gamma_2,\mathcal{A}'')$, suppose that $(W,[\xi],f)$ is the boundary of the bordism $((Z,W),[\tilde{\xi}],\tilde{f})$. Then 
$$(M \times Z, M \times W, \Psi(F, [\tilde{\xi}]), g\times \tilde{f})$$
gives a bordism which has boundary $<(M,F,g),(W,[\xi],f)>$. 
\par
That the pairing respects vector bundle modification (in both groups) follows from
$$<(M,F,g)^{V_M},(W,[\xi],f)>= <(M,F,g),(W,[\xi],f)>^{p_1^*(V_M)}$$
and 
$$<(M,F,g),(W,[\xi],f)^{V_W}>=<(M,F,g),(W,[\xi],f)>^{p_2^*(V_W)}$$
where $V_M$ and $V_W$ are spin$^c$-vector bundles with even rank over $M$ and $W$ respectively and $p_1$ and $p_2$ are the projections from $M\times W$ to the first and (respectively) second factor.
\end{proof}

\section{Higher $\eta$-type invariants on $\mathcal{S}_*^{geo}(\Gamma,C_\epsilon^*)$}
\label{subsectionhigherrho}
In this section, we study $\eta$-type invariants on $\mathcal{S}^{geo}_*(\Gamma,C^*_\epsilon)$ in the case of $C^*$-completions. This construction is similar in spirit to the material in \cite{HReta} (also see \cite{PS}). The reason that we restrict to the $C^*$-algebraic framework is due to the fact that bivariant $K$-theory for $C^*$-algebras is more well understood and many more tools are available, e.g. techniques such as those applied in \cite{MeyNest}. It is not to be expected that these constructions restrict to the $C^*$-algebraic world; however a more explicit understanding of the involved invariants are needed for an analogous Banach algebraic construction. 

Let $C^*_\epsilon(\Gamma)$ denote a $C^*$-algebra completion of $\C[\Gamma]$. The reader is encouraged to keep the reduced or the universal $C^*$-algebra completion in mind. Using an inductive limit argument, we can, without loss of generality, assume that  $B\Gamma$ is a finite CW-complex. We use the notation $\mathcal{L}_\epsilon:=E\Gamma\times_\Gamma C^*_\epsilon(\Gamma)\to B\Gamma$ for the Mishchenko bundle constructed from the completion $C^*_\epsilon(\Gamma)$. Let $N$ denote a ${\rm II}_1$-factor and $\phi:\C\to N$ the inclusion of the unit. Given $C^*$-algebras, $A$ and $B$, let $E(A,B)$ denote the unitary isomorphism classes of Kasparov cycles (see \cite[Definition 2.2]{Kasp}). The input data required for our construction is the following:

\begin{define}
Let $D$ be a unital $C^*$-algebra and $\aleph_0=(E_1,E_2,\varphi)$ be a cocycle in $K_1(C(B\Gamma)\otimes D;\rz)$; that is, $E_1$ and $E_2$ are $D$-bundles over $B\Gamma$ and $\varphi: E_1\otimes_{\phi} N \rightarrow E_2\otimes_{\phi} N$ is a isomorphism of $D\otimes N$-bundles  (see \cite{Bas, Kar} for more on these cocycles). Also let $\phi_1$ and $\phi_2$ be $*$-homomorphisms, $\phi_i: C^*_\epsilon(\Gamma) \rightarrow D$. We assume the following compatiblity between the cocycle and the $*$-homomorphisms. For $i=1$ and $2$, let $h_i \in E(\field{C}, D\otimes C(B\Gamma)\otimes C([0,1]))$ defining equivalences (see \cite[Definition 2.3]{Kasp}) between the Kasparov cycles $(\Gamma(E_i),0) \in E(\field{C},D\otimes C(B\Gamma))$ and $(\Gamma(\mathcal{L}_{\epsilon} \otimes_{\phi_i} D),0) \in E(\field{C},D\otimes C(B\Gamma))$ respectively; the existence of such homotopies is an assumption. Summarizing, the input data is $\aleph=((E_1,E_2,\varphi),\phi_1,\phi_2,h_1,h_2)$.  \label{inputData}
\end{define}

\begin{ex}
\label{sigmaonetwoexample}
The prototypical example of such data comes from the theory of finite dimensional unitary representations. In this example, we must assume that we are working with the full group $C^*$-algebra. Let $\sigma_1$ and $\sigma_2$ be finite dimensional unitary representation of the group $\Gamma$ of rank $k$. In this case, $D=M_k(\field{C})$. The bundles, $E_1$ and $E_2$, are not the flat vector bundles usually considered in such constructions; in our case, the relevant bundles are (by definition)
$$E_i = E\Gamma \times_{\sigma_i} M_k(\field{C}) $$
These bundles have matrix fibers, but the construction is equivalent to the standard construction of flat vector bundle (i.e., $E\Gamma \times_{\sigma_i} \field{C}^k$). Furthermore, $E_1$ and $E_2$ are flat and of the same rank. As such, they are isomorphic after tensoring over the ${\rm II}_1$-factor. The two representations induce $*$-homomorphisms and natural choices for the homotopies (i.e., $h_1$ and $h_2$) follow from \cite[Lemma 4.2]{RWY}.
\end{ex}

One can generalize the previous example to a pair unitary representations of the group into the unitary group of a unital $C^*$-algebra. In this case, the existence of an isomorphism $\varphi$ is an assumption on the particular representations. Further, one can also consider almost representations (see \cite{Dal, MM} for more on these objects).
\par
From a given choice of input data $\aleph=((E_1,E_2,\varphi),\phi_1,\phi_2,h_1,h_2)$, we construct a map $\ind^\mathbb{R}_\aleph:\mathcal{S}^{geo}_*(\Gamma,C^*_\epsilon) \rightarrow K_*^{geo}(pt;D\otimes N)$ which fits into a certain exact sequence (see Theorem \ref{exactSeqForRmap}). Before constructing this map, we have the following proposition which, like the previous example, motivates the definition of the input data (see Definition \ref{inputData}). If $\aleph_0=(E_1,E_2,\varphi)$ is as in Definition \ref{inputData}, we define 
\begin{equation}
\label{rmodzforaleph}
\ind_{\aleph_0}^{\rz}:=<(E_1,E_2,\varphi),  \cdot >:K_*^{geo}(B\Gamma)\to K_{*+1}(D;\rz),
\end{equation}
where $<(E_1,E_2,\varphi),  \cdot >$ denotes the pairing between $K_1(C(B\Gamma)\otimes D;\rz)$ and $K^*(C(B\Gamma))$ under the isomorphism $K^*(C(B\Gamma))\cong K_*^{geo}(B\Gamma)$ (see for example \cite[Section 6]{DeeRZ}).

\begin{prop} \label{BCandBockMapComm}
We use the notation of the previous paragraphs (in particular, $\aleph$ and $\aleph_0$ are as in Definition \ref{inputData}). Then, the following diagram commutes
\begin{center}
$\begin{CD}
K_*^{geo}(B\Gamma) @>\mu_\epsilon>> K_*^{geo}(pt;C^*_\epsilon(\Gamma))   \\
 @V \ind_{\aleph_0}^{\rz} VV  @VV (\phi_1-\phi_2)_* V  \\
  K_{*+1}(D;\rz) @>\delta>> K_*^{geo}(pt;D) 
\end{CD}$
\end{center}
where $\delta$ is the map in the Bockstein sequence associated to $K$-theory with coefficients in $\rz$.
\end{prop} 
\begin{proof}
Before beginning the proof, we note that the maps in the diagram are defined at the level of geometric cycles. On the one hand, for the assembly map and $\phi_1-\phi_2$ this is clear. On the other hand, the definition of Bockstein map and pairing at the level of geometric cycles can be found in \cite{DeeRZ}. 
\par
Let $(M,E,f)$ be a cycle in $K_*^{geo}(B\Gamma)$. Then, based on \cite{DeeRZ}, we have that 
$$\delta( < (E_1,E_2,\varphi),  (M,E,f)>)  =  (M,[E\otimes_{\field{C}} f^*(E_1)]-[E\otimes_{\field{C}} f^*(E_2)])$$
and
\begin{eqnarray*}
(\phi_1-\phi_2)(\mu(M,E,f)) & = & (\phi_1-\phi_2)(M,E\otimes_{\field{C}} f^*(\mathcal{L}_{\epsilon})) \\
& = & (M,[E \otimes_{\field{C}} f^*(\mathcal{L}_{\epsilon})\otimes_{\phi_1}D]-[E \otimes_{\field{C}} f^*(\mathcal{L}_{\epsilon})\otimes_{\phi_2}D]) \\
& = & (M,[E \otimes_{\field{C}} f^*(\mathcal{L}_{\epsilon}\otimes_{\phi_1}D)]-[E \otimes_{\field{C}} f^*(\mathcal{L}_{\epsilon}\otimes_{\phi_2}D)])
\end{eqnarray*}
That these two cycles give the same class in $K_*^{geo}(pt;D)$ follows from the existence of the homotopies $h_1$ and $h_2$.
\end{proof}
The next proposition requires us to fix the following notation. Let $D$ be a unital $C^*$-algebra, $N$ a ${\rm II}_1$-factor, and $\phi:D \rightarrow D\otimes N$ be the unital $*$-homomorphism given by $a \mapsto a\otimes I$. Also let $\Gamma$ be a discrete group, $W$ a compact spin$^c$-manifold with boundary, $f:\partial W \rightarrow B\Gamma$ a continuous map, and $\aleph=((E_1,E_2,\varphi),\phi_1,\phi_2,h_1,h_2)$ be a choice of input data as in Definition \ref{inputData}. We will require the following maps between $K$-theory groups:
\begin{enumerate}
\item $\sigma_{C^*_\epsilon(\Gamma)}: K^0(W;C^*_\epsilon(\Gamma)) \oplus K^0(\partial W) \rightarrow K^*(\partial W;C^*_\epsilon(\Gamma))$ is the map induced from the map (defined from bundles to $K$-theory classes) defined via
$$(\mathcal{E}_{C^*_\epsilon(\Gamma)}, E ) \mapsto [\mathcal{E}_{C^*_\epsilon(\Gamma)}|_{\partial W}]-[E\otimes_{\field{C}} f^*(\mathcal{L}_{\epsilon})] $$ 
\item Let $\pi:W\times [0,1] \rightarrow W$ be the projection map. Define 
\begin{align*}
\gamma: &K^0(W;C^*_\epsilon(\Gamma)) \oplus K^0(\partial W) \rightarrow\\
&\qquad\rightarrow K^0(W;D) \oplus K^0( \partial W \times [0,1], \partial W \dot{\cup} \partial W; \phi) \oplus K^0(W;D)
\end{align*}
to be the map induced from the map at the level of bundles defined via
\begin{align*}
 (&\mathcal{E}_{C^*_\epsilon(\Gamma)},E_{\field{C}}) \mapsto\\
  \bigg(\mathcal{E}_{C^*_\epsilon(\Gamma)}\otimes_{\phi_1}&D, \left(\pi^*(E)\otimes_{\field{C}} E_1 \otimes_{\phi} N, E\otimes_{\field{C}}E_1 \dot{\cup} E\otimes_{\field{C}}E_2, \mathrm{id} \dot{\cup}\mathrm{id}\otimes \varphi\right) ,\mathcal{E}_{C^*_\epsilon(\Gamma)}\otimes_{\phi_2}D\bigg)\end{align*}
\item Let $\psi : K^*(\partial W;C^*_\epsilon(\Gamma)) \rightarrow K^0(\partial W;D) \oplus K^0(\partial W;D)$ be the map induced from the map on bundles given by 
$$ \mathcal{E}_{C^*_\epsilon(\Gamma)} \mapsto (\mathcal{E}_{C^*_\epsilon(\Gamma)} \otimes_{\phi_1} D , \mathcal{E}_{C^*_\epsilon(\Gamma)} \otimes_{\phi_2} D)$$
\item Finally, we let 
\begin{align*}
\zeta: K^0(W;D) \oplus K^0( \partial W \times [0,1], \partial W \dot{\cup} \partial W; &\phi) \oplus K^0(W;D) \rightarrow\\
& K^0(\partial W;D) \oplus K^0(\partial W;D) 
\end{align*}
be the map induced from the map (from cocycles to $K$-theory classes) defined via
$$(F_1,(V_1,V_2,\varphi_V),F_3 ) \mapsto ([F_1|_{\partial W}] - [V_1] ,[F_3|_{\partial W}] - [V_2] )$$
\end{enumerate}
\begin{prop} \label{compDiaIsComm}
 We use the notation introduced in the discussion preceeding this proposition. Then, the following diagram is commutative:
\[
\small
\xymatrix@C=1em@R=2.71em{
 K^0(W;C^*_\epsilon(\Gamma)) \oplus K^0(\partial W)\ar[dd]^{\gamma} \ar[r]^{\sigma} & K^0(\partial W;C^*_\epsilon(\Gamma)) \ar[d]^{\psi} \\
    &K^0(\partial W;D) \oplus K^0(\partial W;D)  \\
     K^0(W;D) \oplus K^0( \partial W \times [0,1],\partial W \dot{\cup} \partial W;\phi) \oplus K^0(W;D)\ar[ur]^{\zeta} & } 
\normalsize
\]
\end{prop}
\begin{proof}
For simiplicity, we consider the relevant maps for bundles; the general case of difference classes follows from this one. Let $(\mathcal{E}_{C^*_\epsilon(\Gamma)},E_{\field{C}})$ be (respectively) a $C^*_\epsilon(\Gamma)$-bundle over $W$ and a vector bundle over $\partial W$; then, its image under $\psi \circ \sigma$ is 
$$\left([(\mathcal{E}_{C^*_\epsilon(\Gamma)}\otimes_{\phi_1} D)|_{\partial W}] - [E_{\field{C}} \otimes f^{*}(E_1) ],[\mathcal{E}_{C^*_\epsilon(\Gamma)}\otimes_{\phi_2}D)|_{\partial W}] - [E_{\field{C}} \otimes f^{*}(E_2)]\right)$$
While, its image under $\zeta \circ \gamma$ is 
$$ \left([\mathcal{E}_{C^*_\epsilon(\Gamma)}|_{\partial W}\otimes_{\phi_1} D] - [E_{\field{C}} \otimes f^*(\mathcal{L}_{\epsilon})\otimes_{\phi_1}D) ],[\mathcal{E}_{C^*_\epsilon(\Gamma)}|_{\partial W}\otimes_{\phi_2} D] - [E_{\field{C}} \otimes f^*(\mathcal{L}_\epsilon)\otimes_{\phi_2}D) \right)$$
The existence of the homotopies $h_1$ and $h_2$ implies that these two classes are equal.  
\end{proof}
In fact, during the proof of the proposition, we have proved more than just that the diagram commutes; we have constructed an explicit homotopy. Using this homotopy and methods in \cite{MeyNest}, we can complete the diagram:
\[
\scriptsize
\xymatrix@C=1em@R=2.71em{
K^0(W,\partial W;\mu_{\mathcal{L}_\epsilon})\ar[r] \ar[d]_{\beta}& K^0(W;C^*_\epsilon(\Gamma)) \oplus K^0(\partial W)\ar[dd]^{\gamma} \ar[r]^{\sigma} & K^0(\partial W;C^*_{\epsilon}(\Gamma)) \ar[d]^{\psi} \\
K^0(Z, W \dot{\cup} W;\phi)\ar[dr] &    &K^0(\partial W;D) \oplus K^0(\partial W;D)  \\
& K^0(W;D) \oplus K^0( \partial W \times [0,1],\partial W \dot{\cup} \partial W;\phi) \oplus K^0(W;D)\ar[ur]^{\zeta} & } 
\normalsize
\]
 where $Z=W\cup \partial W\times [0,1]\cup -W$. We denote the map completing the diagram by $\beta$. We also obtain a map from $K^0(W,\partial W;\mu_{\mathcal{L}_\epsilon})$ to $K^0(Z;D\otimes N)$ by taking the composition of $\beta$ with the map $K^0(Z, W \dot{\cup} W;\phi) \rightarrow K^0(Z;D\otimes N)$; we denote this map  by $\alpha$. For a cycle $(W,\xi,f)$ in $\mathcal{S}^{geo}(\Gamma,C^*_\epsilon)$, we define
\begin{equation} \label{defnOfind}
\ind_{\aleph}^\mathbbm{R} (W,\xi,f):= [Z, \alpha(\xi)]\in K_*^{geo}(pt;D\otimes N).
\end{equation}
\begin{prop}
The mapping $\ind_{\aleph}^\mathbbm{R}: \mathcal{S}^{geo}_*(\Gamma,C^*_\epsilon) \rightarrow K_*^{geo}(pt;D\otimes N)$ is well-defined.
\end{prop}
\begin{proof}
We begin with the bordism relation. Suppose that $(W,\xi,f)$ is the boundary of $((Z,W),\nu,g)$. Using results in \cite{CFPerMap}, we can form $Y$ -- a smooth compact spin$^c$ manifold with boundary, homeomorphic to $Z \cup_{\partial Z - {\rm int}{W}} W\times [0,1] \cup_{\partial Z - {\rm int}{W}} -Z$. Moreover, we can assume that the boundary of $Y$ is $(W\cup_{\partial W} \partial W \times[0,1]\cup_{\partial W} -W)$. We can apply more or less the same argument as in the preceeding paragraphs to obtain a map 
$$\alpha_Z : K^0(Z,\partial Z- {\rm int}{W};\mu_{\mathcal{L}_\epsilon}) \rightarrow K^0(Y;D\otimes_{\phi}N)$$
such that $\alpha_Z|_W = \alpha$. It follows that $(Y, \alpha_Z(\nu))$ is a bordism (with respect to $K_*^{geo}(pt;D\otimes N)$); it has boundary $((W\cup \partial W \times[0,1]\cup -W,\alpha(\xi))$. This completes the argument for the bordism relation.
\par
The proof for the vector bundle modification relation proceeds as follows. Let $(W,\xi,f)$ be a cycle in $\mathcal{S}^{geo}(\Gamma,C^*_\epsilon)$ and $V$ a spin$^c$ vector bundle of even rank over $W$. Let $W^V$ denote the manifold in the cycle obtained by taking the vector bundle modification of $(W,\xi,f)$ by $V$.  It follows from the fact that the Thom isomorphism is a natural isomorphism that the following diagram commutes:
\begin{center}
$\begin{CD}
K^0(W,\partial W;\mu_{\mathcal{L}_\epsilon}) @>>> K^*(W^V,\partial W^V;\mu_{\mathcal{L}_\epsilon})   \\
 @V\alpha VV  @V \alpha^{V} VV  \\
K^0(W;D\otimes_{\phi}N) @>>> K^0(W^V;D\otimes_{\phi}N)
\end{CD}$
\end{center}
where 
\begin{enumerate}
\item the horizontal maps are obtained from the Thom isomorphism;
\item $\alpha^V$ denotes the map (on $K^0(W^V,\partial W^V;\mu_{\mathcal{L}_\epsilon})$) obtained from the input data $((E_1,E_2,\varphi),\phi_1,\phi_2,h_1,h_2)$ with $W^V$ (rather than $W$) in the construction of $\alpha$;
\end{enumerate} 
This commutative diagram implies that the map respects the vector bundle relation.
\end{proof}

\begin{theorem} \label{exactSeqForRmap}
Let $\aleph=((E_1,E_2,\varphi),\phi_1,\phi_2,h_1,h_2)$ and $\aleph_0$ be as in Definition \ref{inputData}. Then, the following diagram commutes
\begin{center}
\scriptsize
$\begin{CD}
@>>> K_*^{geo}(B\Gamma) @>\mu_\epsilon>> K_*^{geo}(pt;C^*_{\epsilon}(\Gamma))  @>>> \mathcal{S}^{geo}_*(\Gamma,C^*_\epsilon) @>>> K_{*-1}^{geo}(B\Gamma) @>>>  \\
 @. @V \ind_{\aleph_0}^{\rz} VV  @V (\phi_1-\phi_2)_* VV @V {\rm ind}_\aleph^\mathbbm{R}VV @V \ind_{\aleph_0}^{\rz} VV  \\
@>>>  K_{*+1}(D;\rz) @>\delta>> K_*^{geo}(pt;D) @>\phi_*>> K_*^{geo}(pt;D\otimes N) @>r>> K_*(D;\rz) @>>>
\end{CD}$
\normalsize
\end{center}
where 
\begin{enumerate}
\item the first exact sequence is the one constructed in Theorem \ref{sixTerExtSeq};
\item the second exact sequence is the Bockstein sequence associated to $K$-theory with $\rz$-coefficients;
\item the first vertical mapping $\ind_{\aleph_0}^{\rz}$ is defined above in \eqref{rmodzforaleph} and the second vertical mapping $ {\rm ind}_\aleph^\mathbbm{R}$ is defined in \eqref{defnOfind}.
\end{enumerate}
\end{theorem}
\begin{proof}
Proposition \ref{BCandBockMapComm} implies that we need only show that
\begin{center}
$\begin{CD}
K_*^{geo}(pt;C^*_\epsilon(\Gamma))  @>>> \mathcal{S}^{geo}_*(\Gamma,C^*_\epsilon) @>>> K_{*-1}^{geo}(B\Gamma)  \\
@V \phi_1-\phi_2 VV @V  {\rm ind}_\aleph^\mathbbm{R}VV @V  {\rm ind}_{\aleph_0}^{\rz}VV  \\
K_*^{geo}(pt;D) @>\phi_*>> K_*^{geo}(pt;D\otimes N) @>r>> K_*(D;\rz) 
\end{CD}$
\end{center}
commutes. 
\par
For the first square, commutativity is clear: a cycle in $\mathcal{S}_*^{geo}(\Gamma,C^*_\epsilon)$ whose manifold data has empty boundary is mapped to exactly $\phi_1-\phi_2$ by $ {\rm ind}_\aleph^\mathbbm{R}$. 
\par
For the second square, let $(W,\xi,f)$ be a cycle in $\mathcal{S}_*^{geo}(\Gamma,C^*_\epsilon)$. Then 
$$<(E_1,E_2,\varphi), \delta (W,\xi,f)> = (\partial W \times [0,1], \eta) $$
where $\eta$ is an element of $K^0(\partial W\times [0,1],\partial W \dot{\cup} -\partial W;\phi)$; this group and the construction of $\eta$ can be found in \cite{DeeRZ}.
Let $Z=W\cup \partial W \times [0,1] \cup -W$. Then, using the map $\beta:K^0(W,\partial W;\mu) \rightarrow K^0(Z,W\dot{\cup} W;\phi)$ defined in the discussion following Proposition \ref{compDiaIsComm}, we have an element $\nu \in K^0(Z\times [0,1], (W \cup -W)\times \{1\};\phi)$ (namely, $\beta(\xi)$) such that 
$$ \nu|_{W\times \{0\}}=\alpha(\xi) \hbox{ and } \nu|_{\partial W \times [0,1]}=\eta $$
Thus, $(Z\times [0,1], \nu)$ defines a bordism (in $K_*(D;\rz)$) between 
$$<(E_1,E_2,\varphi), \delta (W,\xi,f)>$$ 
and $\delta\circ {\rm ind}(W,\xi,f)$. This completes the proof of the commutativity of the second square.
\end{proof}

\begin{cor}
Let $\Gamma$ be a torsion-free, discrete group. Assuming that the completion $C^*_\epsilon(\Gamma)$ satifies the Baum-Connes conjecture, i.e., $\mu_{\epsilon}$ is an isomorphism, the mapping 
$$\ind_{\aleph}^\mathbbm{R}: \mathcal{S}^{geo}_*(\Gamma,C^*_\epsilon) \rightarrow K_*^{geo}(pt;D\otimes N)\quad\mbox{vanishes.}$$
\end{cor}

\appendix

\section{Geometric $K$-homology with coefficients in a Banach algebra}
\label{banachktheory}

In this appendix, we will prove various results for geometric $K$-homology with coefficients in a Banach algebra. The results proven in the same way as for $C^*$-algebra coefficients we only sketch the proofs of. The geometric cycles with coefficients in a Banach algebra were defined in Definition \ref{geometriccyclesbacoeff}. 

Note that the association $X\mapsto K_*^{geo}(X;B)$ is a contra variant functor from the category of locally compact Hausdorff spaces with morphisms being continuous functions. This follows directly from the definition because if $g:X\to Y$ is a continuous mapping, there is an induced mapping 
$$g_*:K_*^{geo}(X;B)\to K_*^{geo}(Y;B),$$
defined on cycles by $g_*(M,\mathcal{E}_B,\varphi)=(M,\mathcal{E}_B,g\circ \varphi)$.

\begin{prop}
\label{hominvariance}
If $G:X\times [0,1]\to Y$ is a continuous mapping, and we set $g_t(x):=G(x,t)$, then
$$(g_0)_*=(g_1)_*:K_*^{geo}(X;B)\to K_*^{geo}(Y;B).$$
\end{prop}

\begin{proof}
If $(M,\mathcal{E}_B,\phi)$ is a cycle for $K_*^{geo}(X;B)$, consider the cycle with boundary $(M\times [0,1],\mathcal{E}_B\times[0,1],G\circ (\varphi\times\mathrm{id}_{[0,1]}))$. The boundary of this cycle is exactly $(g_1)_*(M,\mathcal{E}_B,\varphi)\dot{\cup}(g_0)_*(-M,\mathcal{E}_B,\varphi)$ proving that 
$$(g_1)_*(M,\mathcal{E}_B,\varphi)\sim_{bor}(g_0)_*(M,\mathcal{E}_B,\varphi).$$
\end{proof}

\begin{lemma}
\label{analyticassemblyc}
The topological index mapping 
$$\ind_B^t:K_*^{geo}(pt;B)\to K_*(B), \quad (M,\mathcal{E}_B)\mapsto \beta_N i_![\mathcal{E}_B],$$ 
defined in Remark \ref{anassem}, is an isomorphism of abelian groups.
\end{lemma}

\begin{proof}
We will prove that $\mathfrak{p}:K_*(B)\to K_*^{geo}(pt;B)$, $\xi\mapsto (pt,\xi)$ is an inverse to $\ind_B^t$. It is clear that $\ind_B^t\circ \mathfrak{p}=\mathrm{id}_{K_*(B)}$. It remains to prove that $(M,\mathcal{E}_B)\sim (pt,\beta_N i_![\mathcal{E}_B])$ in $K_*^{geo}(pt;B)$. One can prove that $(M,\mathcal{E}_B)\sim (S^{2N},i_![\mathcal{E}_B])\sim (pt, \beta_N i_![\mathcal{E}_B])$, mutatis mutandis to Theorem $4.1$ of \cite{BOOSW}.
\end{proof}

If $Y\subseteq X$ is a closed subset, we define a relative $B$-cycle for $(X,Y)$ to be a triple $(M,\mathcal{E}_B,f)$ where
\begin{enumerate}
\item $M$ is a smooth spin$^c$-manifold with boundary.
\item $f:M\to X$ is a continuous mapping such that $f(\partial M)\subseteq Y$.
\item $\mathcal{E}_B\to M$ is a locally trivial smooth bundle of finitely generated projective $B$-modules.
\end{enumerate}
We let $K_*^{geo}(X,Y;B)$ denote the abelian group of relative $B$-cycles for $(X,Y)$ modulo disjoint union/direct sum, vector bundle modification and bordism. We note that
\[K_*^{geo}(X,Y;B):=\varinjlim_{X'\subseteq X\mbox{   compact}} K_*^{geo}(X',X'\cap Y;B).\]
A direct consequence of this identity guarantees additivity of $K_*^{geo}(-;B)$ in the sense that if $X=\coprod_\alpha X_\alpha$ then
\begin{equation}
\label{additivitykhom}
K_*^{geo}(X;B)=\bigoplus _{\alpha} K_*^{geo}(X_\alpha;B).
\end{equation}

A mapping $g:(X,Y)\to (X',Y')$ between two pairs $(X,Y)$ and $(X',Y')$ as above is a continuous mapping $g:X\to X'$ such that $g(Y)\subseteq Y'$. Associated to such a map, there is an induced map 
$$g_*:K_*^{geo}(X,Y;B)\to K_*^{geo}(X',Y';B),$$
defined on cycles by $g_*(M,\mathcal{E}_B,\phi)=(M,\mathcal{E}_B,g\circ \phi)$. The proof of the next proposition follows in the same manner as Proposition \ref{hominvariance} after ``straightening the angle", see \cite[Lemma $4.6.1$]{Rav}.

\begin{prop}
\label{homoinvofrelative}
Assume that $(X,Y)$ and $(X',Y')$ are two pairs as above and that the two mappings $h,g:(X,Y)\to (X',Y')$ are homotopic, then 
\[g_*=h_*:K_*^{geo}(X,Y;B)\to K_*^{geo}(X',Y';B).\]
\end{prop}

\begin{lemma}
\label{exactnesslemma}
For any pair $(X,Y)$ as above, there is a boundary mapping $\partial:K_*^{geo}(X,Y;B)\to K_{*+1}^{geo}(Y;B)$, defined on cycles by 
$$\partial(M,\mathcal{E}_B,\varphi)=(\partial M,\mathcal{E}_B|_{\partial M},\varphi|_{\partial M}),$$
making the following sequence exact
$$\cdots\xrightarrow{\partial} K_*^{geo}(Y;B)\xrightarrow{i_*}K_*^{geo}(X;B)\to K_*^{geo}(X,Y;B)\xrightarrow{\partial}K_{*+1}^{geo}(Y;B)\to\cdots,$$
Here $i:Y\to X$ denotes the inclusion.
\end{lemma}

The proof of Lemma \ref{exactnesslemma} is similar to the one in \cite[Proposition 4.6.8]{Rav}; it also goes along the same lines as the proof of Theorem \ref{sixTerExtSeq}. The proof of the next lemma is also similar to a proof in \cite{Rav}; in this case, the relevant proposition is 4.6.7 of \cite{Rav}.

\begin{lemma}
\label{excisionlemma}
If $U\subseteq X$ is an open subset such that $\overline{U}\subseteq A^\circ$, then the natural mapping 
\[K_*^{geo}(X\setminus U,A\setminus U;B)\to K_*^{geo}(X,A; B),\]
is an isomorphism.
\end{lemma}

Combining \eqref{additivitykhom}, Proposition \ref{homoinvofrelative}, Lemma \ref{exactnesslemma} and Lemma \ref{excisionlemma}, we have the following theorem.

\begin{theorem}
\label{kgeohomtheory}
$K_*^{geo}(-;B)$ is a generalized homology theory.
\end{theorem}

As a direct consequence of this theorem, there is an abstract characterization of Banach algebras for which there exists an analytic index mapping $\ind_{X,B}^a:K_*^{geo}(X;B)\to KK^{ban}(C_0(X),B)$ and it being an isomorphism. It is unclear to the authors when this condition is fulfilled beyond the case $B=\C$. Let $\chi\in C^\infty(\R)$ be a function such that $\chi(t)=1$ for $t>>0$ and $\chi(t)=-1$ for $t<<0$. We identify a smooth bundle of finitely generated projective $B$-modules $\mathcal{E}_B\to M$ with a projection $p_{\mathcal{E}_B}\in C^\infty(M;M_N(B))$, for some $N$, via Serre-Swan's theorem. The analytic index is given on the level of cycles by 
$$\ind_{X,B}^a:(M,\mathcal{E}_B,\phi)\mapsto \phi_*\left[p_{\mathcal{E}_B}L^2(M;B^N),p_{\mathcal{E}_B}(\chi(D_{M}))\otimes \mathrm{id}_{B^N})p_{\mathcal{E}_B}\right],$$
where $D_{M}$ denotes the spin$^c$-Dirac operator on $M$. If $KK^{ban}(-,B)$ is half-exact on the category of Banach algebras of continuous functions vanishing at infinity on a finite $CW$-complex, the mapping $\ind_{X,B}^a$ respects the bordism relation. We conclude the following Corollary of Theorem \ref{kgeohomtheory}.

\begin{cor}
\label{kkbankkgeo}
If $B$ is a Banach algebra such that the contra variant functor $KK^{ban}(-,B)$ is half-exact on the category of Banach algebras of continuous functions vanishing at infinity on a finite $CW$-complex, the natural transformation
\[\ind_{X,B}^a:K_*^{geo}(X;B)\to KK^{ban}(C_0(X),B),\]
is a well-defined isomorphism for any finite $CW$-complex $X$.
\end{cor}

\begin{remark}
If $B$ is a $C^*$-algebra making $KK^{ban}(-,B)$ half-exact, Corollary \ref{kkbankkgeo} implies that 
$$KK(C_0(X),B)\cong KK^{ban}(C_0(X),B)$$
for any finite $CW$-complex $X$. The isomorphism is the natural, forgetful mapping sending a Kasparov cycle to a $KK^{ban}$-cycle.
\end{remark}

\paragraph{\textbf{Acknowledgements}}
The authors wish to express their gratitude towards Heath Emerson, Nigel Higson, and Thomas Schick for discussions. The first listed author was supported by an NSERC postdoctoral fellowship. Both authors thank the Courant Centre of G\"ottingen, the Leibniz Universit\"at Hannover and the Graduiertenkolleg 1463 (\emph{Analysis, Geometry and String Theory}) for facilitating this collaboration.

\end{document}